\documentclass[12pt, a4paper, reqno]{amsart}

\usepackage{amsmath, amssymb, amsthm}
\usepackage{tikz}
\usepackage[T1]{fontenc}
\usepackage[utf8]{inputenc}
\usepackage[all]{xy}
\usepackage[
  margin=1.5cm,
  includefoot,
  footskip=30pt,
  ]{geometry}
\usepackage{xcolor}
\usepackage[colorlinks=true,citecolor=cyan,linkcolor=magenta, urlcolor=violet, filecolor=green, backref=page]{hyperref}
\usepackage{hyperref}
\usepackage{indentfirst, caption}

\newcommand{\cR}{\mathcal{R}}

\newcommand{\bC}{\mathbb{C}}
\newcommand{\bF}{\mathbb{F}}

\newcommand{\bQ}{\mathbb{Q}}

\newcommand{\bZ}{\mathbb{Z}}

\newcommand{\Qbar}{\overline{\mathbb{Q}}}

\newcommand{\ff}{\mathfrak{f}}

\newcommand{\fp}{\mathfrak{p}}

\newcommand{\fs}{\mathfrak{s}}

\DeclareMathOperator{\Aut}{Aut}

\DeclareMathOperator{\Gal}{Gal}
\DeclareMathOperator{\GL}{GL}
\DeclareMathOperator{\Hom}{Hom}
\DeclareMathOperator{\Jac}{Jac}

\DeclareMathOperator{\Twist}{Twist}

\DeclareMathOperator{\disc}{disc}

\DeclareMathOperator{\ord}{ord}

\DeclareMathOperator{\tame}{tame}
\DeclareMathOperator{\wild}{wild}

\theoremstyle{plain}
\newtheorem{theorem}{Theorem}[section]
\newtheorem{lemma}[theorem]{Lemma}
\newtheorem{proposition}[theorem]{Proposition}

\newtheorem{definition}[theorem]{Definition}

\theoremstyle{definition} 
\newtheorem{remark}[theorem]{Remark}

\newtheorem{conjecture}{Conjecture}

\theoremstyle{remark} 

\DeclareFontFamily{U}{wncy}{}
\DeclareFontShape{U}{wncy}{m}{n}{<->wncyr10}{}
\DeclareSymbolFont{mcy}{U}{wncy}{m}{n}
\DeclareMathSymbol{\Sha}{\mathord}{mcy}{"58}

\newcommand{\lp}{\left(}
\newcommand{\rp}{\right)}

\newcommand{\lara}[1]{\left\langle #1 \right\rangle}
\newcommand{\lbrb}[1]{\lp #1 \rp}
\newcommand{\lcrc}[1]{\left\{ #1 \right\}}

\newcommand{\quadsym}[2]{\lbrb{\frac{#1}{#2}}}

\title{Goldfeld conjecture for non-hyperelliptic direction}
\author{Keunyoung Jeong}
\address{Department of Mathematics Education, Chonnam National University, 77, Yongbong-ro, Buk-gu, Gwangju 61186, Korea}
\email{keunyoung@jnu.ac.kr}

\author{Junyeong Park}
\address{Department of Mathematics Education, Chonnam National University, 77, Yongbong-ro, Buk-gu, Gwangju 61186, Korea}
\email{junyeongp@gmail.com}

\begin{document}

\begin{abstract}
Since the curve $y^2 = x^6+1$ has a large automorphism group, there exist twist families arising from non-hyperelliptic directions.
In this paper, we give an explicit upper bound on the average analytic rank of such a family, assuming the generalized Riemann hypothesis for the $L$-functions.  Also, we propose an analogue of the Goldfeld conjecture for the family following Katz--Sarnak philosophy.
\end{abstract}

\maketitle

\section{Introduction}

In \cite{Gol79}, Goldfeld conjectured that the distribution of analytic ranks of quadratic twists of the given elliptic curve is governed by the so-called minimalist conjecture.
Let $E$ be an elliptic curve defined by an equation $y^2 = f(x)$ and let $E^{(k)}$ be a quadratic twist defined by
\begin{align*}
    E^{(k)} : ky^2 = f(x), \qquad k \in \bQ^\times/(\bQ^\times)^2.
\end{align*}
Then, Goldfeld's conjecture predicts that 50\% of $E^{(k)}$ have analytic rank $0$, another 50\% of $E^{(k)}$ have analytic rank $1$, and 0\% but infinitely many have analytic rank $\geq 2$.

Even though this conjecture is widely believed to be true and there has been substantial progress on the conjecture \cite{Smi1, Smi2}, a naive analogue of the Goldfeld conjecture for the elliptic curves over number fields or for genus $2$ curves does not hold.
For example, Dokchitser--Dokchitser \cite{DD09} found an elliptic curve over a number field whose quadratic twists always have the root number $+1$.
Klagsbrun--Mazur--Rubin gave a theoretical explanation of the disparity by studying the Selmer group, and suggested a generalization of the Goldfeld conjecture for elliptic curves over a number field \cite[Conjecture 7.12]{KMR13}.

Let $C$ be a hyperelliptic curve defined by the equation $y^2=f(x)$. For $k\in\mathbb{Q}/(\mathbb{Q}^\times)^2$, the curve $C^{(k)}$ defined by $ky^2=f(x)$ is called the hyperelliptic twist of $C$, because it is determined by the hyperelliptic involution and the quadratic extension $\mathbb{Q}(\sqrt{k})/\mathbb{Q}$.
A naive analogue of the Goldfeld conjecture predicts that the distribution of analytic ranks of $C^{(k)}$ follows the 50\%-50\%-0\% rule. 
However, Dokchitser--Dokchitser \cite[Lemma 4.1]{DD16} gave an example of a genus $2$ curve whose hyperelliptic twists have root number $-1$. 
Yu \cite[Theorem 1.1]{Yu16} and Morgan \cite[Theorem 1.2]{Mor19} generalized the result of \cite[Theorem A]{KMR13} to the Jacobian of odd-degree hyperelliptic curves and to principally polarized abelian varieties.
Consequently, their results led to an analogous conjecture for the Jacobians of odd-degree hyperelliptic curves and for principally polarized abelian varieties, in the spirit of \cite[Conjecture 7.12]{KMR13}.

\smallskip 

The main motivation of this paper is to formulate a Goldfeld conjecture for a genus $2$ curve in a \emph{non-hyperelliptic} direction. 
We first note that there are no such directions in the case of elliptic curves.
We recall that 
\begin{align} \label{eqn: twist cohom}
    \mathrm{Twist}(X/\bQ) \cong H^1(G_\bQ, \Aut_{\overline{\bQ}} (X) )
\end{align}
where $\mathrm{Twist}(X/\bQ)$ is the set of twists of $X$ over $\bQ$. 
Since the automorphism group of an elliptic curve is generically (i.e., if the $j$-invariant is not equal to $0$ or $1728$) isomorphic to the cyclic group of order $2$, we have
\begin{align*}
    \mathrm{Twist}(E/\bQ) \cong H^1(G_\bQ, \lcrc{\pm 1}) = \Hom(G_\bQ, \lcrc{\pm 1}) \cong \bQ^\times/(\bQ^\times)^2.
\end{align*}
So the twist of a given elliptic curve $E$ must be one of the quadratic twists $E^{(k)}$.
Therefore, to consider the different directions of twists, we should consider curves of higher genus with larger automorphism groups.

However, even if we choose a suitable curve with a large automorphism group and consider a non-hyperelliptic twist family, it is insufficient in the view of arithmetic statistics for two reasons. First, if a given twist family of a higher genus curve contains many hyperelliptic twists, then the statistical parameter (e.g., average rank) is largely affected by the arithmetic of hyperelliptic twists.
Therefore, if we want to study the arithmetic statistics on the twists family of a given curve in a different direction, rather than hyperelliptic twists, we should consider a twist family, which is a family of twists where no two distinct elements in the family are hyperelliptic twists of each other.
Second, even if a twist family of hyperelliptic curves that are not hyperelliptic twists of each other, the rank distribution of the family sometimes turns out to be the known problem if the Jacobian of each twist is not simple over a base field.

\smallskip 

By \eqref{eqn: twist cohom}, it is natural to consider a curve with a large automorphism group in order to explore ``another direction''.
Our choice is the curve $C_0: y^2 = x^6 + 1$ over $\bQ$, whose automorphism group has order $24$, which is somewhat the most complicated one in the case of genus $2$ (cf. \cite[Table 1]{Car} and \cite[Table 3]{CL}).
Cardona--Lario \cite{CL} gave a classification of isomorphism classes of twists of the curve.
Let 
\begin{align} \label{eqn: type list}
    \mathcal{T} = \lcrc{I, C_2^A, \cdots, C_2^G, C_3, V_4^A, \cdots, V_4^G, \cdots, D_{12}^A, \cdots, C_2 \times D_{12}}
\end{align}
be a set of conjugacy classes of subgroups of $\Aut(\Aut_{\overline{\bQ}}(C_0)) \cong C_2 \times D_{12} $.
From the result of \cite{CL}, we know that
\begin{align*}
    \Twist(C_0/\bQ) = \bigsqcup_{T \in \mathcal{T} } \Twist_T(C_0/\bQ)
\end{align*}
for
\begin{align*}
    \Twist_T(C_0/\bQ) := \lcrc{C \in \Twist(C_0/\bQ) : \Gal(K/\bQ) \cong T \leq \Aut(\Aut_{\overline{\bQ}}(C))  }
\end{align*}
where $K$ is the defining field of $\Aut_{\overline{\bQ}}(C_0)$
(See section \ref{subsec:CL} for more details).
Furthermore, 
the defining equation of an element in $\Twist_T(C_0/\bQ)$ is one of
\begin{align*}
    ky^2 = f_{T, u, v}(x)
\end{align*}
for $k$ in a certain quotient of $\bQ^\times/(\bQ^\times)^2$, and $u, v$ parameters determined by $T \in \mathcal{T}$.
In \cite[\S 7]{CL}, the authors gave a parameter $(u, v)$, in terms of another parameter $d$. 
For example,
\begin{align*}
    (u, v) = \left\{ \begin{array}{lll}
    (1, 1)     & \textrm{if } T = C_2^A \\
    (d, -3)     & \textrm{if } T = V_4^C \\
    (1, d)     & \textrm{if } T = D_{12}^A.
    \end{array} \right.
\end{align*}
In the case where $T=C_2^A$, the $f_{C_{2}^A, 1, 1}(x)$ is uniquely determined and the parameter $d$ does not appear.
Hence, the twist family $\Twist_{C_2^A}(C_0/\bQ)$ consists of hyperelliptic twists.
Also, $\Twist_{C_2^A}(C_0/\bQ)$ is bijective to a quotient of $\bQ^\times/(\bQ^\times)^2$ where the parameter $k$ varies.
So, $\Twist_{C_2^A}(C_0/\bQ)$ is in some sense $1$-dimensional, and it has only the hyperelliptic direction.

Contrary to the $C_2^A$-case, the $V_4^C$-case has two parameters, $k$ and $d$.
This is in some sense $2$-dimensional, and we have a one-parameter family in \emph{$d$-direction} by choosing $k = 1$.
This family is
\begin{align*}
    y^2 &=54(d^3 + 1) x^6 + 324( d^3 - 1) x^5 + 810(d^3 + 1) x^4 + 1080(d^3 - 1) x^3 \\
        &+ 810(d^3 + 1) x^2 + 324(d^3 - 1) x + 54( d^3 + 1),
\end{align*}
and it satisfies that the elements are not hyperelliptic twists of each other.
The authors and Kwon studied this family in \cite{JKP24}. 
However, their Jacobians are always $\bQ$-isogenous to products of elliptic curves.
Therefore, the rank distribution problem for such twists in \cite{JKP24} is reduced to the Goldfeld conjecture for elliptic curves over $\bQ$,\footnote{Unfortunately, the proof of \cite{JKP24} should be modified. See Remark \ref{rmk: JKP}. }
and the average rank would be $ \frac{1}{2} + \frac{1}{2} = 1$.

The main object of this paper comes from $D_{12}^A$ with $k = 1$.
The equation is given by
\begin{align*}
C_d : y^2 &= \frac{27(d-3)}{d+3}\left(x^6-\frac{12d}{d-3}x^5-5dx^4+\frac{40d^2}{3(d-3)}x^3+\frac{5d^2}{3}x^2-\frac{4d^3}{3(d-3)}x-\frac{d^3}{27}\right)
\end{align*}
when $d \neq 0, \pm 3$.
We will show that no two distinct elements in $\{C_d\}$ are hyperelliptic twists of one another (cf. Proposition \ref{prop: Kd and Cd notwist}).
Therefore, it is reasonable to say that the twist family $\lcrc{C_d}$ is a family of twists of $C_0$ along a non-hyperelliptic direction.
A notable difference compared with the $V_4^C$-case is that the Jacobian of $C_d$ is always simple (cf. Proposition \ref{prop: simple}).
We remark that the defining equation of $C_d$ may not be canonical, but our choices are made uniformly in $d$; see Remark \ref{rmk: choices} for details.

Since the Jacobian $J_d$ of $C_d$ has complex multiplication, the $L$-function of $C_d$ admits an analytic continuation and satisfies the functional equation. 
We denote the analytic rank of $C_d$ by $g_d$.
In this paper, we give an explicit upper bound for the average analytic ranks of the family $C_d$ under the generalized Riemann hypothesis.

Let $S(X)$ be the set of positive square-free integers $d \neq 1, 3$ such that $|d| <  X$.

\begin{theorem} \label{thm: AAR}
Suppose the generalized Riemann hypothesis for the $L$-function of $C_d$ for every $d \in S(X)$.
For every fixed $\sigma > 0$,
\begin{align*}
\frac{1}{|S(X)|} \sum_{d \in S(X)} g_d 
    &\leq \frac{1}{2} +  (6+o(1))\frac{1}{\sigma} + O\lbrb{ \frac{1}{\log X} +  \frac{X^{\frac{3\sigma}{2}- \frac{1}{2}}}{\log X} + \frac{X^{\frac{\sigma}{2} - \frac{1}{2} }   }{\log X}    }.
\end{align*}
\end{theorem}

Therefore, we have an upper bound $18 + \frac{1}{2}$, which is larger than the expectation.
The Katz--Sarnak philosophy predicts that the limiting one-level density should hold without the restriction of the support.
Following the philosophy, we conjecture the following.
\begin{conjecture}
    The average analytic rank of $\lcrc{C_d}$ for positive square-free integers $d$ is $\frac{1}{2}$.
\end{conjecture}

The main technique is estimating the $1$-level density using the explicit formula.
To apply it, we need to control the weighted sum of the trace of Frobenius $a_p(C_d)$ and $a_{p^2}(C_d)$.
Here, the work of Fit\'e--Sutherland \cite{FS} which express $a_p(C_d)$ and $a_{p^2}(C_d)$ in terms of $a_p(E_0)$ where $E_0 : y^2 = x^3 + 1$, is crucially used.
On the other hand, to apply the explicit formula, we need to determine the conductor exponents.
For this, we use recent advances in the local arithmetic of hyperelliptic curves due to Dokchitser--Dokchitser--Maistret--Morgan \cite{DDMM}, together with the user's guide \cite{BBB}.

We conclude the introduction by giving an outline of this article.
In section \ref{sec: pre}, we recall the results of Cardona--Lario \cite{CL} and Fit\'e--Sutherland \cite{FS} and give some remarks on the normalization.
In section \ref{sec: arith inv}, we compute some necessary arithmetic invariants. Using this information, we give a proof of the main theorem in section \ref{sec: proof of main}.
In this section, we also consider the other types, especially the type $D_{12}^B$.
But in this case, we find that the problem becomes less interesting since the twists are not simple, as the $V_4^C$-type twists studied in \cite{JKP24}.

\bigskip 
\textbf{Acknowledgment.} 
The authors thank Dohyeong Kim for the useful discussion.
Keunyoung Jeong was supported by the National Research Foundation of Korea (NRF) grant funded by the Korea government(MSIT) (RS-2024-00341372 and RS-2024-00415601).
Junyeong Park was supported by Basic Science Research Program through the National Research Foundation of Korea(NRF) funded by the Ministry of Education (RS-2024-00449679) and the National Research Foundation of Korea(NRF) grant funded by the Korea government(MSIT) (RS-2024-00415601).

\section{Preliminaries}  \label{sec: pre}

\subsection{Review of Cardona--Lario} \label{subsec:CL}

It is well known that the automorphism group of a genus $2$ curve over a field of characteristic zero is one of 
\begin{align*}
    C_2, \quad V_4, \quad D_8, \quad D_{12}, \quad C_{10}, \quad \widetilde{S_4}, \quad 2D_{12}
\end{align*}
where $C_n$ (resp. $D_n$) is the cyclic group (resp. dihedral group) of order $n$, $V_4$ is the Klein $4$ group, $\widetilde{S_4}\cong\GL_2(\bF_3)$ is a double cover of the symmetric group $S_4$, and $2D_{12}$ is a double cover of $D_{12}$.
We recall that $\overline{\bQ}$-isomorphism between genus $2$ curves can be represented as 
\begin{align*}
    (x', y') = \lbrb{\frac{mx + n}{px + q}, y \frac{mq-np}{(px + q)^3}}, \qquad \begin{pmatrix}
        m & n  \\ p & q
    \end{pmatrix} \in \GL_2(\overline{\bQ}).
\end{align*}
Especially, any automorphism group of a given genus $2$ curve is a $G_{\bQ}$-subgroup of $\GL_2(\overline{\bQ})$.

Cardona and his collaborators studied a parametrization of twists of a genus $2$ curve with a specified automorphism group.
Cardona--Quer \cite{CQ} gave a classification of twists of a genus $2$ curve whose automorphism group is isomorphic to $D_8$ or $D_{12}$, and Cardona \cite{Car} and Cardona--Lario \cite{CL} studied the case of $\widetilde{S_4}$, and $2D_{12}$, respectively.
In this paper, we mainly use the result of \cite{CL}.

Let $C/\bQ$ be a twist of $C_0 : y^2 = x^6+1$.
Then the automorphism group $\Aut_{\overline{\bQ}}(C)$ is isomorphic to $\mathrm{Aut}_{\overline{\bQ}}(C_0)$ as an abstract group, but the $G_{\bQ}$-group structures are not the same in general.
The first step of the classification of twists of $C_0$ is classifying the $G_{\bQ}$-group which is $2D_{12}$ as a group, and is embedded into $\GL_2(\overline{\bQ})$.
For a group $A$, the $G_{\bQ}$-action on $A$ can be described by
\begin{align*}
\xymatrix{G_{\bQ} \ar@{->>}[r] & \Gal(K/\bQ) \ar[r]^-{\cong} & H \subset \Aut(A)  }
\end{align*}
where $K$ is the field of definition of the action.
In the case of a twist of $C_0$, we know that 
\begin{align*}
    \Aut(\Aut_{\overline{\bQ}}(C_0)) = \Aut(2D_{12}) \cong  C_2 \times D_{12}
\end{align*}
as a group.
Therefore, $H$ is one of possible subgroups of $C_2 \times D_{12}$. 
We note that for each $\phi : G_{\bQ} \to \Aut(A)$ and $\rho \in \Aut(A)$,
the conjugate homomorphism
\[
\phi^\rho(\sigma) := \rho \circ \phi(\sigma) \circ \rho^{-1},  \qquad \sigma \in G_\bQ
\]
induces the same $G_{\bQ}$-action on $A$ (cf. \cite[p.197]{CL}).
Hence, we only need to consider conjugacy classes of subgroups.
For example, there are six conjugacy classes in $C_2 \times D_{12}$
that are isomorphic to $D_{12}$, which we denote by
$D_{12}^A, \dots, D_{12}^F$ following \cite{CL}.

For a given $H$, the isomorphism between $\Gal(K/\bQ)$ and $H$ can be described by specifying subfields of $K$.
For example, there are $6$ isomorphisms when $H \cong V_4$.
Let $\iota, \jmath$ be generators of $V_4^A$ .
Then, giving an isomorphism between $\mathrm{Gal}(K/\mathbb{Q})$ and $H$ is equivalent to giving the two quadratic fields fixed by $\langle\iota\rangle$ and $\langle\jmath\rangle$.
Similarly, for any $H \leq C_2 \times D_{12}$, the isomorphism between $\Gal(K/\bQ)$ and $H$ can be distinguished by specifying three (at most) quadratic fields.
A full description of these quadratic fields and the list of subgroups $H$ are given in \cite[Table 3, pp.200-201]{CL}.

Summarizing the discussion so far, a $G_{\bQ}$-group $A$ which is isomorphic to $2D_{12}$ as an abstract group can be characterized by a conjugacy class of subgroups $H\subseteq C_2 \times D_{12}$, $K$ the field definition of the action, and three (at most) quadratic fields
\begin{align} \label{eqn: def Ki}
    K_1 = \bQ(\sqrt{u}), \qquad 
    K_2 = \bQ(\sqrt{v}), \qquad 
    K_3 = \bQ(\sqrt{v'}).
\end{align}
In other words, $u, v, v' \in \bQ^\times$, but they could be square.

We now introduce additional notation to describe the field $K$ and the defining equation of $C$.
Let $\widetilde{K}$ be the field such that $K_1, K_2 \subset \widetilde{K} \subset K$ and $\Gal(\widetilde{K}/\bQ)$ is the one labeled as $D_{12}$-type in \cite[Table 3]{CL}.
The existence of $\widetilde{K}$ and its properties are described in \cite[p.202]{CL}, and we note that if $\Gal(K/\bQ) \cong D_{12}$, then $3 \mid [\widetilde{K} : \bQ]$.
The following is the restatement of \cite[Remark 1]{CL}.

\begin{definition} \label{def:uvsz}
(i) Suppose that $3 \nmid [\widetilde{K}:\bQ]$.
If $u,v \in \bQ$ satisfy $(u, -3v) = 1 \in \mathrm{Br}_2(\bQ)$, then
\begin{align*}
    x^2 + \frac{3}{v}y^2 = u.
\end{align*}
has solutions in $\mathbb{Q}^\times$. Once solutions $\alpha,\beta\in\mathbb{Q}^\times$ are chosen, we denote
\begin{align*}
    z := 4 \alpha^3 - 3u\alpha, \qquad 
    s := \frac{u^2 - 2\alpha^2u + \alpha z}{3\beta}.
\end{align*}
(ii) Suppose that $ 3 \mid [\widetilde{K} : \bQ]$. If $(u, -3v) = 1 \in \mathrm{Br}_2(\bQ)$, then
\begin{align*}
    u^3 - z^2 = 3s^2v
\end{align*}
for some $z \in \bQ$ and $s \in \bQ^\times$. 
Also, we define $\alpha$ as a zero of 
\begin{align*}
    x^3 - \frac{3u}{4}x - \frac{z}{4}
\end{align*}
and $\beta := \frac{u^2-2\alpha^2u + \alpha z}{3s}$.
\end{definition}

We will use case (ii) in this paper.
Recall that our first goal is to determine which $G_{\mathbb{Q}}$-groups can be embedded in $\mathrm{GL}_2(\overline{\mathbb{Q}})$.

\begin{theorem} \label{thm:CLthm2}
Let $A$ be a $G_{\bQ}$-group with underlying group isomorphic to $2D_{12}$. Let $K$ be the defining field of $G_{\bQ}$-group structure on $A$, and let $K_i$ be the subfield of $K$ defined by (\ref{eqn: def Ki}) for $i=1, 2, 3$.

Then $A$ can be embedded in $\GL_2(\Qbar)$ if and only if $(u, -3v) = 1 \in \mathrm{Br}_2(\bQ)$ and $v' \equiv -3v \pmod{\bQ^{\times 2}}$.
In this case, $A \cong \lara{U, V, -1}$ where
\begin{align*}
    U = \frac{1}{\sqrt{u}}
    \begin{pmatrix}
    \alpha & \beta \\ \frac{3\beta}{v} & -\alpha 
    \end{pmatrix}, \qquad
    V =  \frac{\sqrt{-3}}{2}
    \begin{pmatrix}
    1 & \frac{\sqrt{v}}{3} \\ -\frac{1}{\sqrt{v}} & 1
    \end{pmatrix}, \qquad 
    -1 = \begin{pmatrix}
        -1 & 0 \\ 0 & -1
    \end{pmatrix}
\end{align*}
and the field of definition $K$ is $\bQ(\alpha, \sqrt{u}, \sqrt{v}, \sqrt{-3})$.
\end{theorem}
\begin{proof}
This is \cite[Theorem 2]{CL}.
\end{proof}

For such an $A$, we can find a defining equation of a genus $2$ curve whose automorphism group is $G_{\bQ}$-isomorphic to $A$.

\begin{theorem} \label{thm:CLProp34}
For a $G_{\bQ}$-group $A$ embedded in $\GL_2(\Qbar)$ and isomorphic to $2D_{12}$ as an abstract group, there is a twist defined by the equation
\begin{align*}
    y^2 = 27zx^6 - 162svx^5 - 135vzx^4 + 180sv^2x^3 + 45v^2zx^2 - 18sv^3x - v^3z
\end{align*}
whose automorphism group is $G_{\bQ}$-isomorphic to $A$.
Furthermore, two twists whose automorphism groups are $G_{\bQ}$-isomorphic to $A$ differ by a hyperelliptic twist.
\end{theorem}
\begin{proof}
This is \cite[Propositions 3, 4]{CL}.
We note that $z$ and $s$ are given by Definition \ref{def:uvsz}.
\end{proof}

For a twist $C/\bQ$, the $G_{\bQ}$-group structure on $\Aut_{\overline{\bQ}}(C)$ is determined by a conjugacy class of subgroup of $C_2 \times D_{12}$, and $u, v, v' = -3v \in \bQ^\times/(\bQ^\times)^2$.
A twist $C$ of $C_0$ is said to be of \emph{type $T$} if the conjugacy class of $C_2 \times D_{12}$ is $T$.
The set of twists of type $T$ is denoted by $\Twist_T(C_0/\bQ)$ in the introduction.

Using the above two theorems, we can classify the twists of $C_0/\bQ$.

\begin{theorem} \label{thm: CL classification}
The set $\Twist(C_0/\bQ)$ of $\mathbb{Q}$-isomorphism classes of twists of $C_0$ is in bijection with the set of tuples $(A,k)$ such that
\begin{enumerate}
    \item $A$ is a $G_\mathbb{Q}$-group whose underlying group is isomorphic to $2D_{12}$, and whose Galois action is determined by the field of definition $K$ and its quadratic subfields $K_1=\mathbb{Q}(\sqrt{u})$, $K_2=\mathbb{Q}(\sqrt{v})$, and $K_3=\mathbb{Q}(\sqrt{v'})$.
    \item $k\in(\mathbb{Q}^\times/(\mathbb{Q}^\times)^2))/S$ where $S=\langle u,v'\rangle$ if $3\nmid[K:\mathbb{Q}]$ and $S=\langle v'\rangle$ if $3\mid[K:\mathbb{Q}]$.
\end{enumerate}
For each pair $(A, k)$, the defining equation of the corresponding twist $C_0$ is $ky^2 = f(x)$ where $f(x)$ is given by Theorem \ref{thm:CLProp34}.
\end{theorem}

Recall that there are six conjugacy classes of subgroups in $C_2\times D_{12}$ isomorphic to $D_{12}$, denoted by $D_{12}^A,\cdots,D_{12}^F$ following \cite{CL}. In particular, by \cite[\S7]{CL}, the Brauer group condition is trivial for $D_{12}^A$ and $D_{12}^B$. In this paper, we will focus on these two types.

\subsection{Review of Fit\'e--Sutherland}
Let $C$ over $\bQ$ be a genus $2$ curve with good reduction at $p$.
It is known that the zeta function
\begin{align*}
    \zeta(C/\bF_p, s) := \exp \lbrb{\sum_{m=1}^{\infty} \frac{|C(\bF_{p^m})|}{m}p^{-ms}}
\end{align*}
is a rational function. More precisely,
\begin{align*}
    \zeta(C/\bF_p, s) = \frac{P_1(p^{-s})}{(1-p^{-s})(1 - p^{1-s})}
\end{align*}
where
\begin{align*}
    P_1(t) =  \prod_{i=1}^4(1 - \alpha_{p, i}t) \in \bZ[t]
\end{align*}
for some algebraic integers $\alpha_{p, i}$.
By the Weil conjecture, we have $|\alpha_{p,i}|=\sqrt{p}$.

Let $a_{p, i}(C)$ be the $i$-th coefficient of $P_1(t)$, i.e.,
\begin{align*}
    P_1(t) =  1 +  a_{p, 1}(C)t + a_{p, 2}(C)t^2 + a_{p, 3}(C)t^3+ p^2 t^4.   
\end{align*}
Then we have
\begin{align*}
    a_{p, 1}(C) = -\sum_{i=1}^4 \alpha_{p, i}, \qquad 
    a_{p, 2}(C) = \sum_{i>j} \alpha_{p, i} \alpha_{p, j}.
\end{align*}
The functional equation of $\zeta(C/\bF_p, s)$ shows that $a_{p, 3}(C) = pa_{p, 1}(C)$.


The $L$-factor $L_p(C/\mathbb{Q},s)$ of $C$ at a good prime $p$ is defined to be $P_1(p^{-s})^{-1}$. Then the $L$-function of $C$ is defined to be the product of $L$-factors. Define $a_n(C)$ by the following relation:
\begin{align*}
L(C/\bQ, s) = \prod_p L_p(C/\bQ, s) = \sum_{n=1}^{\infty} \frac{a_n(C)}{n^s}.
\end{align*}
Later, we will also use
\begin{align*}
    \lambda_{p^k}(C) := \frac{a_{p^k}(C)}{p^{\frac{k}{2}}}, \qquad 
    b_{p^k}(C) := \sum_{i=1}^4 \frac{\alpha_{p, i}^k}{p^{\frac{k}{2}}}
\end{align*}
to follow the convention of Iwaniec--Kowalski \cite{IK} in section \ref{sec: proof of main}.
We note that $b_p(C) = -a_{p, 1}(C)/\sqrt{p}$ and $b_{p^2}(C) = (a_{p, 1}(C)^2 - 2a_{p, 2}(C))/p$.

For a prime $p$ where $C$ has good reduction, 
\begin{align*}
    L_p(C/\bQ, s) 
    &=  \prod_{i=1}^4(1 - \alpha_{p, i}p^{-s})^{-1} = 1 + \lbrb{\sum_{i=1}^4 \alpha_{p, i}} p^{-s} + \lbrb{  \sum_{i=1}^4 \alpha_{p, i}^2  +  \sum_{i > j}\alpha_{p, i} \alpha_{p, j}  }p^{-2s} + \cdots.
\end{align*}
Hence 
\begin{align} 
    a_p(C) &= \sum_{i=1}^4 \alpha_{p, i} = -a_{p, 1}(C), \nonumber  \\
    a_{p^2}(C) &= \sum_{i=1}^4 \alpha_{p, i}^2  +  \sum_{i > j}\alpha_{p, i} \alpha_{p, j}
    = \lbrb{\sum_{i=1}^4 \alpha_{p, i}}^2 - \sum_{i > j} \alpha_{p, i} \alpha_{p, j} = a_{p, 1}(C)^2 - a_{p, 2}(C). \label{eqn: ap2 ap1 ap2}
\end{align}

We note that Fit\'e--Sutherland \cite{FS} use a different normalization of Fourier coefficients.
They define
\begin{align*}
L_p(A, t) := \prod_{i=1}^{2g} (1 - \alpha_it),
\quad \text{where} \quad
|A(\bF_{p^k})| = \prod_{i=1}^{2g}(1- \alpha_i^k)
\end{align*}
for $n \geq 0$ and $g = \dim A$, and define
\begin{align*}
        \bar{L}_p(A, t) := L_p(A, t/\sqrt{p}) \qquad 
        \sum_{i=0}^{2g} a_i(A)(p)t^i = \bar{L}_p(A, t).
\end{align*}
In other words, if $J$ is the Jacobian of $C$, then $a_i(J)(p)$ in \cite{FS} corresponds to our $a_{p,i}(J)/p^{\frac{i}{2}}$, since the $L$-function of a curve and its Jacobian coincide.

When $A = E$ is an elliptic curve, $a_{1}(E)(p)$ in \cite{FS} corresponds to $-a_p(E)/\sqrt{p}$, where $a_p(E)$ is the usual trace of Frobenius. 
Therefore, the statements
\begin{align*}
    a_1(\Jac(C))(p) = a_1(E)(p), \qquad a_2(\Jac(C))(p) = a_1(E)(p)^2 + 2
\end{align*}
in \cite{FS} can be translated into our notation as follows: 
\begin{align*}
    a_{p, 1}(C) = -a_p(E), \qquad 
    a_{p, 2}(C) = a_p(E)^2 + 2p.
\end{align*}

Fit\'e--Sutherland \cite[Proposition 4.9]{FS} showed that when $J$ is the Jacobian of twists of the curve $C_0$, its $L$-factor at good prime $p >3$ is determined by
\begin{align*}
    I(p) := (f_L(p), f_K(p), f_M(p)) 
    \quad \textrm{and} \quad
    a_p := a_p(E_0)
\end{align*}
where $E_0 : y^2 = x^3 + 1$ is an elliptic curve.
Here $L$ is the number field where the isomorphisms between $C$ and $C_0$ are defined, $K$ is the number field where the automorphisms of $C$ are defined which coincides with the notation $K$ of the previous section \ref{subsec:CL}, $M = \bQ(\sqrt{-3})$, and $f_{\bullet}(p)$ is the residue degree of $p$ in the extension field $\bullet$ of $\bQ$. We remark that \cite{FS} also gave similar results for the twists of $y^2 = x^5 - x$.

\begin{proposition} \label{prop: list Lp}
Let $J$ be the Jacobian of a twist $C$ of $C_0$ over $\bQ$, and let $p > 3$ be a prime of good reduction for $C$.
Then,
\begin{align*}
	\begin{array}{|c|c|c|c|}
	\hline
	I(p) & a_{p,1}(J) & a_{p,2}(J) 
	\\ \hline
	(1,1,1) & -2a_p & a_p^2 + 2p 
	\\
	(2,1,1) & 2a_p & a_p^2 + 2p 
	\\
	(2,2,1) & 0 & -a_p^2 + 2p 
	\\
	(3,3,1) & a_p & a_p^2 - p 
	\\
	(4,2,1) & 0 & a_p^2 - 2p 
	\\
	(6,3,1) & -a_p & a_p^2 - p 
	\\
	(6,6,1) & \pm \sqrt{3(4p - a_p^2)} & -a_p^2 + 5p 
	\\
	(2,2,2) & 0 & 2p 
	\\
	(4,2,2) & 0 & -2p 
	\\
	(6,6,2) & 0 & -p 
	\\
	(12,6,2) & 0 & p 
	\\ \hline
	\end{array}
\end{align*}

\end{proposition}
\begin{proof}
This is    \cite[Proposition 4.9]{FS} for our case.
\end{proof}

Actually, we can show that some rows do not appear in our case, since $\Gal(K/\bQ) \cong S_3 \times C_2$ (cf. Proposition \ref{prop: Kd and Cd notwist}).
As this simple observation already illustrates, to apply this proposition, we need to compute some arithmetic invariants, including Galois groups and conductor exponents. 
We carry out these computations in the next section.

\section{Some arithmetic invariants} \label{sec: arith inv}

\subsection{Parametrization and some invariants}

In this section, we consider twists of type $D_{12}^A$. 
For this type, the parameter $(u, v)$ in Definition \ref{def:uvsz} is $(1, d)$ by \cite[p. 209]{CL}.
Substituting $(u, v) = (1, d)$, into the equation $u^3 - z^2 = 3s^2v$ in Definition \ref{def:uvsz}, we obtain
\begin{align*}
    1 - z^2 = 3ds^2,
\end{align*}
or equivalently, $3ds^2 + z^2 = 1$. Solving the equation gives
\begin{align*}
    (s, z) = \lbrb{\frac{2}{d+3}, \frac{d-3}{d+3}}.
\end{align*}
By Theorem \ref{thm:CLProp34}, we can define a family of the hyperelliptic curves $C_d$ defined by the equation $y^2 = f_d(x)$, where
\begin{align*}
    f_d(x)=\frac{27(d-3)}{d+3}\left(x^6-\frac{12d}{d-3}x^5-5dx^4+\frac{40d^2}{3(d-3)}x^3+\frac{5d^2}{3}x^2-\frac{4d^3}{3(d-3)}x-\frac{d^3}{27}\right)
\end{align*}
for $d \neq 0, \pm 3$. We note that since $d \neq 0$, there is no confusion with $C_0 : y^2 = x^6 + 1$.
We denote by $J_d$ the Jacobian of $C_d$.
Note that $\disc(f_d)=2^{26}3^{21}d^{15}$. By Definition \ref{def:uvsz}, $\alpha_d$ is a root of
\begin{align} \label{eqn: defining alpha}
    x^3-\frac{3}{4}x-\frac{d-3}{4(d+3)}
\end{align}
whose discriminant is
\begin{align*}
    \frac{81d}{4(d+3)^2}.
\end{align*}
Therefore, the splitting field of the polynomial (\ref{eqn: defining alpha}) is $\bQ(\alpha_d, \sqrt{d})$.
Contrary to the previous sections, we attach the subscript $d$ in the definitions of the fields $K, K_i,  L$, and the element $\alpha$. Namely, these will be denoted by $K_d,K_{i,d},L_d$, and $\alpha_d$, respectively.

\begin{proposition} \label{prop: Kd and Cd notwist}
The family $\lcrc{C_d}$ for square-free integers $d$ satisfies the following: \\
(i) $K_d = \bQ(\alpha_d, \sqrt{d}, \sqrt{-3})$. \\
(ii) No two distinct elements of the family are hyperelliptic twists of each other.
\end{proposition}
\begin{proof}
Since we choose $D_{12}^A$-type, $u = 1, v = d, v' = -3d$ and
\begin{align*}
K_{d, 1} = \bQ, \quad  K_{d, 2} = \bQ(\sqrt{d}), \quad  K_{d, 3} = \bQ(\sqrt{-3d}), \quad K_d = \bQ(\alpha_d, \sqrt{d}, \sqrt{-3})
\end{align*}
by Theorem \ref{thm:CLthm2}.
By Theorem \ref{thm:CLProp34}, the $\bQ$-isomorphism class is determined by the above fields and the hyperelliptic twists parameter $k$.
If $d_1$ and $d_2$ are distinct in $\bQ^\times/(\bQ^\times)^2$, then $K_{d_1, 2} \not\cong K_{d_2, 2}$, so $C_{d_1}$ and $C_{d_2}$ are not $\bQ$-isomorphic by Theorem \ref{thm: CL classification}.
Since a hyperelliptic twist does not change the $G_{\bQ}$-group structure on the automorphism group, $C_{d_1}$ and $C_{d_2}$ cannot be hyperelliptic twists of each other.
\end{proof}

\begin{remark} \label{rmk: choices}
Since the solution of the equation $3ds^2 + z^2 = 1$ is not unique,
there may be several possible choices for the equation $f_d$.
Let $C'_d$ be the curve arising from the other choices of solutions $s, z$.
Since $\Aut_{\overline{\bQ}}(C_d) \cong \Aut_{\overline{\bQ}}(C'_d)$ as $G_{\bQ}$-groups, they are hyperelliptic twists of each other, by Theorem \ref{thm: CL classification}.
Hence, our choice of solutions corresponds to the choice of a $\bQ$-isomorphism class
within the hyperelliptic twist family.
In principle, one might worry that such a choice could be tailored so as to bias the average analytic rank toward a particular value.
However, we emphasize that the same construction is applied uniformly for all values of $d$.
In particular, our choice may not be canonical, but it is made in a uniform manner with respect to $d$, and therefore does not involve selecting preferred representatives within each family.
Finally, the ordering of our family is defined in terms of the parameter $|d|$, and hence does not depend on the particular equation used to represent the chosen $\bQ$-isomorphism class.
\end{remark}

\begin{lemma} \label{lem: Q alpha sqrtd S3}
Let $d$ be an integer such that $d \neq 0, \pm 3$. Then, the polynomial (\ref{eqn: defining alpha}) is irreducible.
Furthermore if $d$ is square-free and $d \neq 1$, then $\bQ(\alpha_d, \sqrt{d})/\bQ$ is $S_3$-extension.
\end{lemma}
\begin{proof}
Let $r = a/b$ be a rational solution of the polynomial (\ref{eqn: defining alpha}) such that $a, b$ are relatively prime integers.
Then 
\begin{align*}
    d = -3\frac{4r^3 - 3r + 1}{4r^3 - 3r - 1}= 3 \frac{(b + a)(2a - b)^2}{(b-a)(2a  +b)^2} \in \bZ.
\end{align*}
Suppose that there is an odd prime $p$ such that $p \mid (2a + b)$. 
It does not divide $a$ and $b$. On the other hand,
\begin{align*}
    (2a + b) - 2(a+ b) = -b, \qquad 
    (2a + b) - (2a - b) = 2b
\end{align*}
gives that $p \nmid (a+b), (2a - b)$, which leads a contradiction.
Similarly, if an odd prime $p$ satisfies that $ p \mid (b-a)$ and $p \mid (b+a)$, then $p \mid 2b$, and if $p \mid (b - a)$ and $p \mid (2a - b)$ then $p \mid a$.
So the only possibilities are that
\begin{align*}
    2a + b = \pm 2^k, \qquad b - a = \pm 2^m
\end{align*}
for $k, m \geq 0$. Hence
\begin{align*}
    d &= 3 \frac{ \lbrb{\frac{2}{3}(2a + b) + \frac{1}{3}(b-a)} \lbrb{\frac{1}{3}(2a+b) - \frac{4}{3}(b-a)}^2 }{(b-a)(2a + b)^2}
    = \pm \frac{\lbrb{2^{k+1} \pm 2^{m} } \lbrb{2^k \mp 2^{m+2} }^2}{9 \cdot 2^{2k+m}} \\
    &= \pm \frac{\lbrb{2 \pm 2^{m-k} } \lbrb{1 \mp 2^{m - k+2} }^2}{9 \cdot 2^{m-k}}.
\end{align*}
By comparing the $2$-adic valuation, we know that $d$ is not an integer when $m - k > 1$.
If $m - k = 0, 1$, a direct computation shows that $d \in \lcrc{0, \pm 3}$.
The case $m - k < 0$ is handled in the same way by comparing the $2$-adic valuation.
\end{proof}

\begin{proposition} \label{prop: simple}
Suppose that $d \neq 0, \pm 3$ is a square-free integer. The Jacobian $J_d$ is simple over $\bQ$.
\end{proposition}
\begin{proof}
For a non-simple abelian variety $A \sim B \times C$ and a prime $p$ where $A$ has good reduction, we have $A_{\bF_p} \sim B_{\bF_p} \times C_{\bF_p}$.
Therefore, it suffices to show that $A_{\bF_p}$ is simple for a good prime $p$.

Since $\Gal(K_d/\bQ) \cong S_3 \times C_2$ and $\Gal(M/\bQ) \cong C_2$, we know that there are rational primes $p$ such that $I(p) = (6, 6, 2)$ or $I(p) = (12,6,2)$ for each $d \neq 0, \pm 3$.
By Proposition \ref{prop: list Lp}, the characteristic polynomial of the Frobenius automorphism at $p$ is $T^4 - pT^2 + p^2$ or $T^4 + pT^2 + p^2$, which are irreducible over $\bQ$.
By Honda--Tate theory, $A_{\bF_p}$ is simple if the characteristic polynomial is irreducible. 
Hence, $J_d$ is simple over $\bQ$.
\end{proof}

\begin{lemma} \label{lem: fd splitting field}
Let $F_d$ be the splitting field of $f_d$ over $\bQ$.
Then, $F_d = \mathbb{Q}(\alpha_d,\sqrt{d},\sqrt{3})$.
\end{lemma}
\begin{proof}
We recall that $K_d=\mathbb{Q}(\alpha_d,\sqrt{d},\sqrt{-3})$. By \cite[Remark 5.2]{FS}, we have $K_d\subseteq F_d(\sqrt{-3},\sqrt{-1})$. On the other hand, $f_d$ admits a factorization 
\begin{align}\label{eqn: fd two cubic factor}
    f_d(x) &= \lbrb{x^3-\left(\frac{6d}{d-3}+\frac{d+3}{d-3}\sqrt{3d}\right)x^2-dx+\frac{d}{9}\left(\frac{6d}{d-3}+\frac{d+3}{d-3}\sqrt{3d}\right) } \nonumber \\
    &\times  \lbrb{x^3-\left(\frac{6d}{d-3}-\frac{d+3}{d-3}\sqrt{3d}\right)x^2-dx+\frac{d}{9}\left(\frac{6d}{d-3}-\frac{d+3}{d-3}\sqrt{3d}\right)} 
\end{align}
over $\mathbb{Q}(\sqrt{3d})$ whose discriminants are
\begin{align*}
    \disc\left(x^3-\left(\frac{6d}{d-3}\pm \frac{d+3}{d-3}\sqrt{3d}\right)x^2-dx+\frac{d}{9}\left(\frac{6d}{d-3} \pm \frac{d+3}{d-3}\sqrt{3d}\right)\right)=\frac{16d^3(d+3)^2}{(\sqrt{d}\mp \sqrt{3})^4}.
\end{align*}
Since the splitting field of a cubic over $\bQ$ is the minimal field containing a zero and the square root of the discriminant, we conclude that $\bQ(\sqrt{3d}, \sqrt{d}) = \mathbb{Q}(\sqrt{d},\sqrt{3})\subseteq F_d$. Since $F_d/\mathbb{Q}$ is Galois, the factorization \eqref{eqn: fd two cubic factor} shows that $[F_d:\mathbb{Q}]\leq12$. Moreover, since $K_d \subset F_d(\sqrt{-3}, \sqrt{-1})$, and $\mathbb{Q}(\alpha_d)/\mathbb{Q}$ is cubic by \ref{lem: Q alpha sqrtd S3}, we conclude that $[F_d:\mathbb{Q}]=12$. This also shows that $F_d(\sqrt{-3},\sqrt{-1})=F_d(\sqrt{-1}).$
Combining all these, we get the following tower of finite extensions:
\begin{align*}
    \xymatrix@!0@R=3.5pc@C=7.5pc{
    & F_d(\sqrt{-1}) \ar@{-}[dl]_-2 \ar@{-}[d]^-3 \ar@{-}[dr]^-2 & \\
    K_d \ar@{-}[d]_-4 \ar@{-}[dr]^-3 & \mathbb{Q}(\sqrt{d},\sqrt{3},\sqrt{-1}) \ar@{-}[d]^-2 \ar@{-}[dr]^-2 & F_d \ar@{-}[d]^-3 \\
    \mathbb{Q}(\alpha_d) \ar@{-}[dr]_-3 & \mathbb{Q}(\sqrt{d},\sqrt{-3}) \ar@{-}[d]^-4 & \mathbb{Q}(\sqrt{d},\sqrt{3})\rlap{\ .} \ar@{-}[dl]^-4 \\
    & \mathbb{Q} &
}
\end{align*}
This shows that $F_d(\sqrt{-1})=\mathbb{Q}(\alpha_d,\sqrt{d},\sqrt{3},\sqrt{-1})$ and hence $F_d=\mathbb{Q}(\alpha_d,\sqrt{d},\sqrt{3})$.
\end{proof}


\begin{lemma} \label{lem: quad split of fd}
Let $d \neq 0, 1, \pm 3$ be a square-free integer and let $a, b, c$ be roots of 
\begin{align*}
    x^3+\frac{12d}{d-3}x^2-4dx-\frac{16d^2}{3(d-3)} \in \bQ[x].
\end{align*}
(1) Over $\bQ(\alpha_d, \sqrt{d})$, we have
\begin{align*}
    f_d(x) = \frac{27(d-3)}{d+3} \lbrb{x^2 + ax - \frac{d}{3} } \lbrb{x^2 + bx - \frac{d}{3} } \lbrb{x^2 + cx - \frac{d}{3} }.
\end{align*}
(2) Over $\bQ(\alpha_d)$, we have
\begin{align*}
f_d(x)=\frac{27(d-3)}{d+3}\left(x^2+ax-\frac{d}{3}\right)\left(x^4+(b+c)x^3+\left(bc-\frac{2}{3}d\right)x^2-\frac{d}{3}(b+c)x+\frac{d^2}{9}\right)
\end{align*}
for a suitable $a\in\mathbb{Q}(\alpha_d)$.
\end{lemma}
\begin{proof}
Let  $a, b, c$ be the roots of the polynomial 
\begin{align*}
    x^3+\frac{12d}{d-3}x^2-4dx-\frac{16d^2}{3(d-3)},
\end{align*}
so that 
\begin{align*}
    a + b + c = -\frac{12d}{d-3}, \quad ab + bc + ca = -4d, \quad abc = \frac{16d^2}{3(d-3)}.
\end{align*}
It satisfies
\begin{align*}
    &\lbrb{x^2 + ax - \frac{d}{3}} \lbrb{x^2 + bx - \frac{d}{3}} \lbrb{x^2 + cx - \frac{d}{3}} \\
    &= x^6 + (a + b  + c)x^5 + (-d + ab + bc +ca)x^4 + \lbrb{-\frac{d}{3}(a + b + c) + abc}x^3 \\ 
    & \qquad  +\lbrb{\frac{d^2}{3} - \frac{d}{3}(ab + bc + ca) }x^2 + \frac{d^2}{9}(a +b+c)x - \frac{d^3}{27} \\
    &= x^6-\frac{12d}{d-3}x^5-5dx^4+\frac{40d^2}{3(d-3)}x^3+\frac{5d^2}{3}x^2-\frac{4d^3}{3(d-3)}x-\frac{d^3}{27}.
\end{align*}
Hence, we have a factorization 
\begin{align*}
    f_d(x) = \frac{27(d-3)}{d+3}  \lbrb{x^2 + ax - \frac{d}{3}} \lbrb{x^2 + bx - \frac{d}{3}} \lbrb{x^2 + cx - \frac{d}{3}}
\end{align*}
in $F_d$. Since
\begin{align*}
    \disc \lbrb{x^3+\frac{12d}{d-3}x^2-4dx-\frac{16d^2}{3(d-3)}} = \frac{2^8d^3(d+3)^4}{(d-3)^4},
\end{align*}
the splitting field of the polynomial is $\bQ(\sqrt{d})$, or $S_3$-extension between $\bQ(\sqrt{d})$ and $F_d$, which is exactly $\bQ(\alpha_d, \sqrt{d})$.
Hence (1) follows.

We note that the non-trivial automorphism of $\Gal(\bQ(\sqrt{d})/\bQ)$ permutes two of $a, b, c$, say, $b$ and $c$. Then the factorization over $\bQ(\alpha_d)$ follows.
\end{proof}

\begin{proposition} \label{prop: L bound} $L_d=K_d\left(\sqrt{6(d+3)}\right)=\mathbb{Q}\left(\alpha_d,\sqrt{d},\sqrt{-3},\sqrt{6(d+3)}\right)$.
\end{proposition}
\begin{proof}
Whenever there is an isomorphism between $C_d$ and $C_0$, it induces a bijection between the set of roots of $f_d(x)$ and the set of roots of $x^6+1$. Note that the linear fractional transformations on $\mathbb{C}$ inducing a bijection between the set of roots of $f_d(x)$ and $x^6+1$ are in $\mathrm{PGL}_2(F_d(i))$. Indeed, suppose that a linear fractional transformation corresponding to
\begin{align*}
    \Phi=\begin{pmatrix} \Phi_{11} & \Phi_{12} \\ \Phi_{21} & \Phi_{22} \end{pmatrix}\in\mathrm{PGL}_2(\mathbb{C})
\end{align*}
is determined by triples of distinct points $u_1,u_2,u_3\in\mathbb{C}$ and $v_1,v_2,v_3\in\mathbb{C}$:
\begin{align*}
    \frac{\Phi_{11}u_i+\Phi_{12}}{\Phi_{21}u_i+\Phi_{22}}=v_i,\quad i=1,2,3.
\end{align*}
Since these equations are linear in the entries $\Phi_{ij}$, if there is a number field $F$ containing all the $u_i,v_i$, then we can always choose a representative of $\Phi$ in $\mathrm{GL}_2(F)$. In particular, we may take $F=F_d(i)$ because it contains the roots of both $f_d$ and $x^6+1$.

Next, consider the transformation
\begin{align*}
    \xymatrix{\sigma:\mathbb{P}^1(\mathbb{C}) \ar[r] & \mathbb{P}^1(\mathbb{C}) & z \ar@{|->}[r] & \displaystyle-\frac{d}{3z}\rlap{\ .}}
\end{align*}
This acts as an involution on the set of roots of the polynomial of the form
\begin{align*}
    x^2+ux-\frac{d}{3}\in\mathbb{C}[x]
\end{align*}
for $d\neq0$ because $x=0$ is not a root of this polynomial and
\begin{align}\label{fdquadfactor}
    \left(-\frac{d}{3x}\right)^2+u\left(-\frac{d}{3x}\right)-\frac{d}{3}=-\frac{d}{3x^2}\left(x^2+ux-\frac{d}{3}\right).
\end{align}
Moreover, this defines an automorphism of $C_d$ over $\mathbb{Q}(\sqrt{-3d})$:
\begin{align*}
    \xymatrix{C_d \ar[r]^-\sim & C_d & (x,y) \ar@{|->}[r] & \displaystyle\left(\sigma(x),\left(\frac{1}{x}\sqrt{-\frac{d}{3}}\right)^3y\right)\rlap{\ .}}
\end{align*}
The fixed points of $\sigma$ in $\mathbb{P}^1(\mathbb{C})$ are
\begin{align*}
    z=\pm\sqrt{-\frac{d}{3}}.
\end{align*}
This becomes the root of \eqref{fdquadfactor} if and only if
\begin{align*}
    u=\pm2\sqrt{-\frac{d}{3}}.
\end{align*}
Since $a$, $b$, $c$ in Lemma \ref{lem: quad split of fd} (1) cannot be in $\mathbb{Q}(\sqrt{-3d})$, the $\sigma$ defines a fixed-point-free involution on the set $\mathcal{R}_d\subseteq\mathbb{C}$ of roots of $f_d$. On the other hand, the group of linear fractional transformations preserving the set $\mathcal{R}_0\subseteq\mathbb{C}$ of roots of $x^6+1$ is isomorphic to $D_{12}$ generated by the rotation $z\mapsto\zeta_6z$ and the inversion $z\mapsto z^{-1}$. Recall that an isomorphism from $C_d$ to $C_0$ is represented by a linear fractional transformation corresponding to $\Phi\in\mathrm{PGL}_2(F_d(i))$.

Pushing the involution $\sigma$ along $\Phi$, we get a fixed-point-free involution $\Phi\sigma\Phi^{-1}$ on $\mathcal{R}_0$. There are two types of fixed-point-free involutions on $D_{12}$ up to conjugation:
\begin{itemize}
\item The reflection $z\mapsto z^{-1}$.
\item The rotation $z\mapsto-z$.
\end{itemize}
Suppose first that $\Phi\sigma\Phi^{-1}$ is in the same conjugacy class with $z\mapsto z^{-1}$. Multiplying $\Phi$ by some element in $\mathrm{PGL}_2(F_d(i))$, we may assume that $(\Phi\sigma\Phi^{-1})(z)=z^{-1}$ as a linear fractional transform on $\mathbb{C}$. Note that the fixed point of $z\mapsto z^{-1}$ are $\pm1$. Note that $z\mapsto-z$ induces a valid automorphism on $y^2=x^6+1$. Multiplying $\Phi$ by $z\mapsto-z$ if necessary, we may assume that
\begin{align*}
    \Phi\left(\sqrt{-\frac{d}{3}}\right)=1,\quad\Phi\left(-\sqrt{-\frac{d}{3}}\right)=-1
\end{align*}
each of which is equivalent to
\begin{align*}
    \Phi_{11}\sqrt{-\frac{d}{3}}+\Phi_{12}=\Phi_{21}\sqrt{-\frac{d}{3}}+\Phi_{22},\quad\Phi_{11}\sqrt{-\frac{d}{3}}-\Phi_{12}=-\Phi_{21}\sqrt{-\frac{d}{3}}+\Phi_{22}.
\end{align*}
Then we may write
\begin{align*}
    \Phi(z)=\frac{\displaystyle z+w\sqrt{-\frac{d}{3}}}{\displaystyle wz+\sqrt{-\frac{d}{3}}},\quad w\in F_d(i).
\end{align*}
If $\Phi\in\mathrm{GL}_2(F_d(i))$ induces this transformation, for example
\begin{align*}
    \Phi=\begin{pmatrix} 1 & \displaystyle w\sqrt{-\frac{d}{3}} \\ w & \displaystyle\sqrt{-\frac{d}{3}} \end{pmatrix},
\end{align*}
then $(\det\Phi)\Phi$ defines an isomorphism over $F_d(i)$:
\begin{align*}
    \xymatrix{\displaystyle\left\{y^2=\left(x+w\sqrt{-\frac{d}{3}}\right)^6+\left(wx+\sqrt{-\frac{d}{3}}\right)^6\right\} \ar[r]^-\sim & \left\{y^2=x^6+1\right\}\rlap{\ .}}
\end{align*}
Comparing the ratio of the $x^4$ coefficient to the $x^6$ coefficient of the left hand side and $f_d(x)$, we get
\begin{align*}
    \frac{-5dw^2(w^2+1)}{1+w^6}=-5d\implies\frac{w^2}{w^4-w^2+1}=1\implies(w^2-1)^2=0.
\end{align*}
However, then the matrix $\Phi$ above becomes singular, yielding a contradiction. Therefore, we have $(\Phi\sigma\Phi^{-1})(z)=-z$. Note that the fixed points of $z\mapsto-z$ on $\mathbb{P}^1(\mathbb{C})$ are $0$ and $\infty$. Composing $\Phi$ with
\begin{align*}
    \xymatrix{\left\{y^2=x^6+1\right\} \ar[r]^-\sim & \left\{y^2=x^6+1\right\} & (x,y) \ar@{|->}[r] & \displaystyle\left(\frac{1}{x},\frac{y}{x^3}\right)}
\end{align*}
if necessary, we may assume that
\begin{align*}
    \Phi\left(\sqrt{-\frac{d}{3}}\right)=0,\quad\Phi\left(-\sqrt{-\frac{d}{3}}\right)=\infty.
\end{align*}
each of which is equivalent to
\begin{align*}
    \Phi_{11}\sqrt{-\frac{d}{3}}+\Phi_{12}=0,\quad\Phi_{21}\sqrt{-\frac{d}{3}}-\Phi_{22}=0.
\end{align*}
Then we may write
\begin{align*}
    \Phi(z)=w\frac{\displaystyle z-\sqrt{-\frac{d}{3}}}{\displaystyle z+\sqrt{-\frac{d}{3}}},\quad w\in F_d(i).
\end{align*}
 Again, if $\Phi\in\mathrm{GL}_2(F_d(i))$ induces this transformation, for example
\begin{align*}
    \Phi=\begin{pmatrix} w & \displaystyle -w\sqrt{-\frac{d}{3}} \\ 1 & \displaystyle\sqrt{-\frac{d}{3}} \end{pmatrix},
\end{align*}
then $(\det\Phi)\Phi$ defines an isomorphism over $F_d(i)$:
\begin{align*}
    \xymatrix{\displaystyle\left\{y^2=w^6\left(x-\sqrt{-\frac{d}{3}}\right)^6+\left(x+\sqrt{-\frac{d}{3}}\right)^6\right\} \ar[r]^-\sim & \left\{y^2=x^6+1\right\}\rlap{\ .}}
\end{align*}
Comparing the coefficients of $\frac{d+3}{27(d-3)}f_d$ and
\begin{align*}
    &\quad\frac{1}{1+w^6}\left(w^6\left(x-\sqrt{-\frac{d}{3}}\right)^6+\left(x+\sqrt{-\frac{d}{3}}\right)^6\right)\\
    &=x^6+6\sqrt{-\frac{d}{3}}\frac{1-w^6}{1+w^6}x^5-5dx^4-\frac{20}{3}d\sqrt{-\frac{d}{3}}\frac{1-w^6}{1+w^6}x^3+\frac{5d^2}{3}x^2+\frac{2}{3}d^2\sqrt{-\frac{d}{3}}\frac{1-w^6}{1+w^6}x-\frac{d^3}{27}
\end{align*}
we get the following equations:
\begin{align*}
    6\sqrt{-\frac{d}{3}}\frac{1-w^6}{1+w^6}=-\frac{12d}{d-3},\quad-\frac{20}{3}d\sqrt{-\frac{d}{3}}\frac{1-w^6}{1+w^6}=\frac{40d^2}{3(d-3)},\quad\frac{2}{3}d^2\sqrt{-\frac{d}{3}}\frac{1-w^6}{1+w^6}=-\frac{4d^3}{3(d-3)}
\end{align*}
all of which are equivalent to
\begin{align*}
    \frac{1-w^6}{1+w^6}=\frac{2}{d-3}\sqrt{-3d}\iff w^6=\frac{\displaystyle d-3-2\sqrt{-3d}}{\displaystyle d-3+2\sqrt{-3d}}=\left(\frac{\sqrt{d}-\sqrt{-3}}{\sqrt{d}+\sqrt{-3}}\right)^2.
\end{align*}
Therefore, $w^3\in K_d$. Next, taking $\Phi(x)=\pm i$, we get
\begin{align*}
    x=\frac{w\pm i}{w\mp i}\sqrt{-\frac{d}{3}}
\end{align*}
so, from Lemma \ref{lem: quad split of fd} (1), we get
\begin{align*}
    u:=\left(\frac{w+i}{w-i}+\frac{w-i}{w+i}\right)\sqrt{-\frac{d}{3}}=\frac{2(w^2-1)}{w^2+1}\sqrt{-\frac{d}{3}}\in\{-a,-b,-c\}
\end{align*}
and hence
\begin{align*}
w^2=\frac{2\sqrt{d}-u\sqrt{-3}}{2\sqrt{d}+u\sqrt{-3}}\in K_d.
\end{align*}
This shows that $w=w^3/w^2\in K_d$. Now the ratio of the two leading coefficients gives
\begin{align*}
    \frac{d+3}{27(d-3)}(1+w^6)=\frac{d+3}{27(d-3)}\left(1+\frac{d-3-2\sqrt{-3d}}{d-3+2\sqrt{-3d}}\right)=\frac{2(d+3)}{27\left(\sqrt{d}+\sqrt{-3}\right)^2}.
\end{align*}
Therefore, we conclude $L_d=K_d\left(\sqrt{6(d+3)}\right)$.
\end{proof}

\begin{remark} \label{rmk: JKP}
Let $C : y^2 = f(x)$ be a twist of $C_0$ and let $\gamma_k$ be zeros of $f$.
Then a $\overline{\bQ}$-isomorphism $\phi$ from $C$ to $C_0$ satisfies that
\begin{align*}
    \phi(\gamma_{k_1}) = i, \qquad \phi(\gamma_{k_2}) = -i, \qquad \phi(\gamma_{k_3}) = \zeta_3
\end{align*}
for some $k_1, k_2, k_3 \in \lcrc{1, \cdots, 6}$.
Here \cite[Proposition 3.3]{JKP24} claims that this $\phi$ is defined over $\bQ(\gamma_1, \cdots, \gamma_6, i, \zeta_3)$, but this is not true since $f$ is not monic. 
More precisely, the ratio of $y^2=f(x)$ and the transformation of $y^2=x^6+1$ under $\Phi$ (or $(\det\Phi)\Phi$) is required to be a square in the field we are considering.
\begin{itemize}
    \item For the type C in \cite{JKP24}, we apply Theorem \ref{thm:CLthm2} and Theorem \ref{thm:CLProp34} with
    \begin{align*}
        (u,v,\alpha,\beta)=\left(d,-3,\frac{d+1}{2},\frac{d-1}{2}\right).
    \end{align*}
    Then $K=\mathbb{Q}(\sqrt{d},\sqrt{-3})$, and the defining equation can be rewritten as
    \begin{align*}
        y^2=54d^3(x+1)^6+54(x-1)^6.
    \end{align*}
    Note that, if we set
    \begin{align*}
        \Phi=\begin{pmatrix} \sqrt{d} & \sqrt{d} \\ 1 & -1 \end{pmatrix}\in\mathrm{GL}_2(\mathbb{Q}(\sqrt{d}))
    \end{align*}
    then $(\det\Phi)\Phi$ induces an isomorphism over $\mathbb{Q}(\sqrt{d})$:
    \begin{align*}
        \xymatrix{\left\{y^2=d^3(x+1)^6+(x-1)^6\right\} \ar[r]^-\sim & \left\{y^2=x^6+1\right\}\rlap{\ .}}
    \end{align*}
    Therefore, we have $L=K(\sqrt{6})=K(\sqrt{-2})$.
    \item For the type B in \cite{JKP24}, we apply Theorem \ref{thm:CLthm2} and Theorem \ref{thm:CLProp34} with
    \begin{align*}
        (u,v,\alpha,\beta)=\left(1,d,\frac{d-3}{d+3},\frac{2d}{d+3}\right).
    \end{align*}
    Then $K=\mathbb{Q}(\sqrt{d},\sqrt{-3})$, and the defining equation can be rewritten as
    \begin{align*}
        y^2=\frac{-1}{2(d+3)^3}\left(\left(\sqrt{d}+\sqrt{-3}\right)^6\left(\sqrt{-3}x+\sqrt{d}\right)^6+\left(\sqrt{d}-\sqrt{-3}\right)^6\left(\sqrt{-3}x-\sqrt{d}\right)^6\right).
    \end{align*}
    Note that, if we set
    \begin{align*}
        \Phi=\begin{pmatrix} \left(\sqrt{d}+\sqrt{-3}\right)\sqrt{-3} & \left(\sqrt{d}+\sqrt{-3}\right)\sqrt{d} \\ \left(\sqrt{d}-\sqrt{-3}\right)\sqrt{-3} & -\left(\sqrt{d}-\sqrt{-3}\right)\sqrt{d} \end{pmatrix}\in\mathrm{GL}_2(\mathbb{Q}(\sqrt{-3d})),
    \end{align*}
    then $(\det\Phi)\Phi$ induces an isomorphism over $\mathbb{Q}(\sqrt{-3d})$:
    \begin{align*}
        \xymatrix{\left\{y^2=\left(\sqrt{d}+\sqrt{-3}\right)^6\left(\sqrt{-3}x+\sqrt{d}\right)^6+\left(\sqrt{d}-\sqrt{-3}\right)^6\left(\sqrt{-3}x-\sqrt{d}\right)^6\right\} \ar[r]^-\sim & \left\{y^2=x^6+1\right\}\rlap{\ .}}
    \end{align*}
    Therefore, we have $L=K\left(\sqrt{-2(d+3)}\right)$.
\end{itemize}
\end{remark}

\subsection{Cluster pictures and conductors}

Let $A$ be an abelian variety over $\bQ_p$.
Let $t, u, \delta$ be its toric rank, unipotent rank, and the Swan conductor.
Then the conductor exponent of $A$ over $\bQ_p$ is
\begin{align} \label{eqn: cond exp sum}
    t + 2u + \delta
\end{align}
(cf. \cite[(2.3), p. 247]{BK94}).
Since the Swan conductor is zero when $A$ over $\bQ_p$ has tame reduction, the conductor exponent is bounded by $2 \dim A$ if $A$ over $\bQ_p$ has tame reduction.

\begin{lemma} \label{lem: tame p>3}
Let $J_d$ be the Jacobian of $C_d$ over $\bQ$.
Then $J_d$ has tame reduction at $p$ if $p >3$.
\end{lemma}
\begin{proof}
It is well known that $C_d$ over $\bQ_p$ has tame reduction if and only if $J_d$ over $\bQ_p$ has tame reduction. This is equivalent to the splitting field of $f_d$ being tamely ramified over $\mathbb{Q}_p$ (cf. \cite[Remark 5.7]{BBB}).
By Lemma \ref{lem: fd splitting field}, we have $\Gal(F_d/\bQ) \cong S_3 \times C_2$.
Hence, if a rational prime divides the order of a subgroup of $\Gal(F_d/\bQ)$, it must be $2$ or $3$.
\end{proof}

\begin{proposition} \label{prop: cond bound}
There is an absolute constant $c$ such that the conductor of $C_d$ over $\bQ$ is bounded by $c d^4(d+3)^4$.
\end{proposition}
\begin{proof}
Note that $y^2=f_d(x)$ may not be $\mathbb{Z}_p$-integral for $p \mid(d+3)$. To remedy this, we replace $y$ by $(d+3)^{-1}y$ and hence the given equation by $y^2=(d+3)^2f_d(x)$. For each rational prime $p$, we have
\begin{align*}
    \mathfrak{f}(C_d/\mathbb{Q}_p)\leq\ord_p\disc_{\min}(C_d/\mathbb{Q}_p)\leq\ord_p\disc((d+3)^2f_d)
\end{align*}
where the first inequality comes from \cite[THEOREME 1]{Liu94} and the second from \cite[Proposition 2]{Liu94}. Note that we need a $\mathbb{Z}_p$-integral equation to apply \cite[Proposition 2]{Liu94}. Hence the conductor of $C_d$ divides $\disc((d+3)^2f_d)=2^{26}3^{21}d^{15}(d+3)^{20}$.
Therefore, the radical of the conductor divides $6d(d+3)$.
For $p > 3$, Lemma \ref{lem: tame p>3} gives the bound $4$ for the conductor exponent. 
For $p = 2, 3$, we can use \cite[Theorem 6.2]{BK94} instead.
\end{proof}
We will compute the exact conductor exponent for $p>3$.

We first introduce the concept of cluster picture, following \cite[Definition 3.1, Notation 3.4]{BBB}.
For a given polynomial $f$ over $\bQ_p$, let $\cR$ be the set of roots of $f$.
We say that $\fs \subset \cR$ is a \emph{cluster} if 
\begin{align*}
    \fs = \cR \cap \lcrc{x \in \overline{\bQ_p} : \ord_p(x - z) \geq d}
\end{align*}
for some $z \in \overline{\bQ_p}$ and $d \in \bQ$.
We note that a singleton is also regarded as a cluster.
When $\fs$ is a cluster of size $> 1$, we say that $\fs$ is \emph{proper}.
For a cluster $\fs$ of size $\geq 2$, we define the \emph{depth} by
\begin{align*}
    d_{\fs} := \min_{r, r' \in \fs} \ord_p(r - r').
\end{align*}
If $\fs'$ is a maximal subcluster of $\fs$, we say that $\fs'$ is a \emph{child} of $\fs$. 
For a cluster $\mathfrak{s}$, we denote by $P(\mathfrak{s})$ the cluster such that $\mathfrak{s}$ is its child, and call it the \emph{parent} of $\mathfrak{s}$.
If $\fs$ is a cluster of even cardinality whose children are all even, then we say that $\fs$ is \emph{\"ubereven}.
For a cluster $\fs, \fs'$, we define $\fs \wedge \fs'$ as the smallest cluster containing $\fs$ and $\fs'$.
A cluster is \emph{twin} if its size is $2$, and \emph{cotwin} if it is non-\"ubereven and has a child of size $2g$, where $g$ is the genus of the curve $y^2 = f(x)$.
A cluster $\fs$ is \emph{principal} if $|\fs| < 2g + 2$, proper, not a twin or a cotwin, or if $|\fs| = 2g + 2$, $\fs$ has $\geq 3$ children, and is not a cotwin.
For the examples of cluster pictures, we refer \cite[\S 3]{BBB}.

For a proper cluster $\fs$,  we define
\begin{align*}
    \widetilde{\lambda}_{\fs}
    := \frac{1}{2}\lbrb{\ord_p(c) + \#\lcrc{\text{odd child of }\fs}\cdot d_{\fs} + \sum_{r \not\in \fs}d_{\lcrc{r}\wedge \fs} }.
\end{align*}
For a cluster $\fs$, we denote $I_{\fs}$ as the stabilizer of $\fs$ under $I_{\bQ_p}$, and following \cite[Notation 12.2]{BBB},
\begin{align*}
    \xi_{\fs}(a) := \max \lcrc{-\ord_2([I_{\bQ_p} : I_{\fs}]a), 0}.
\end{align*}
We further define
\begin{align*}
    U &:= \lcrc{\textrm{odd clusters } \fs \neq \cR : \xi_{P(\fs)}(\tilde{\lambda}_{P(\fs)}) \leq \xi_{P(\fs)}(d_{P(\fs)}) }, \\
    V &:= \lcrc{\textrm{proper non-\"ubereven clusters } \fs : \xi_{\fs}(\tilde{\lambda}_{\fs}) = 0 }.
\end{align*}

\begin{theorem} \label{thm: BBB12.3}
Let $n_{\tame}$ and $n_{\wild}$ be the tame and wild parts of the conductor exponent of $y^2 = f(x)$ over $\bQ_p$.
Then,
\begin{align*}
    n_{\mathrm{tame}} &= 2g - \#(U/I_{\bQ_p}) + \#(V/I_{\bQ_p}) + \left\{
    \begin{array}{ll}
    1 & \textrm{if $|\cR|$ and $\ord_p(c)$ are both even, } \\
    0 & \textrm{otherwise;} 
    \end{array}
    \right. \\
    n_{\mathrm{wild}} &= \sum_{r \in \cR/G_{\bQ_p}}\lbrb{\ord_p(\Delta(\bQ_p(r)/\bQ_p)) - [\bQ_p(r):\bQ_p] + f(\bQ_p(r)/\bQ_p) }  
\end{align*}    
\end{theorem}
\begin{proof}
This is \cite[Theorem 12.3]{BBB}.
\end{proof}

From now on, we will compute the cluster pictures of $C_d$.

\begin{lemma} \label{lem: cluster p|d}
Let $p$ be a prime $\geq 5$ such that $\ord_p(d) > 0$.
Then the cluster picture of $f_d$ over $\mathbb{Q}_p$ is $\{\bullet,\cdots,\bullet\}_{\frac{\ord_p(d)}{2}}$.
\end{lemma}
\begin{proof}  
Let $r_1,\cdots,r_6\in\overline{\mathbb{Q}_p}$ be the roots of $f_d$.
Considering the $p$-valuation of the coefficients of $f_d$,
the Newton polygon of $f_d$ over $\mathbb{Q}_p$ is the line segment connecting $(0,3\ord_p(d))$ and $(6,0)$. Consequently, by \cite[Proposition II.6.3]{Neu99}, we have $\ord_p(r_i)=\frac{1}{2}\ord_p(d)$. Moreover, we observe the following.
\begin{itemize}
    \item If $\ord_p(d)$ is even, then $\mathbb{Q}_p(\sqrt{d})/\mathbb{Q}_p$ is unramified.
    \item If $\ord_p(d)$ is odd, then $\mathbb{Q}_p(\sqrt{d})/\mathbb{Q}_p$ is ramified with uniformizer $\sqrt{d}p^{-\frac{\ord_p(d)-1}{2}}$.
\end{itemize}
In either case, denote $\mathfrak{m}_{\mathbb{Q}_p[\sqrt{d}]}\subseteq\mathbb{Z}_p[\sqrt{d}]$ the maximal ideal. Then
\begin{align*}
    \frac{1}{d^3}f_d(\sqrt{d}x)&=\frac{27(d-3)}{d+3} \left(x^6-\frac{12\sqrt{d}}{d-3}x^5-5x^4+\frac{40\sqrt{d}}{3(d-3)}x^3+\frac{5}{3}x^2-\frac{4\sqrt{d}}{3(d-3)}x-\frac{1}{27}\right)\\
    &\equiv-(27x^6-135x^4+45x^2-1)\pmod{\mathfrak{m}_{\mathbb{Q}_p[\sqrt{d}]}}.
\end{align*}
Since $p\geq5$, we have
\begin{align*}
    \disc(27x^6-135x^4+45x^2-1)=2^{10}3^9\neq0
\end{align*}
in the residue field $\mathbb{F}_p$ of $\mathbb{Z}_p[\sqrt{d}]$. Hence
\begin{align*}
    \ord_p\left(\frac{r_i}{\sqrt{d}}-\frac{r_j}{\sqrt{d}}\right)=0
\end{align*}
for distinct $i,j\in\{1,\cdots,6\}$.
\end{proof}

\begin{lemma} \label{lem: inertia action p|d}
Let $p$ be a prime $\geq 5$ with $\ord_p(d) > 0$, and let $\cR = \lcrc{r_i}_{i=1}^6$ be the roots of $f_d$ over $\bQ_p$.
Then, $I_{\bQ_p}$ acts on $\cR$ via the quotient $I_{\mathbb{Q}_p}\rightarrow\mathrm{Gal}(\mathbb{Q}_p(\sqrt{d})/\mathbb{Q}_p)$.
If $\ord_p(d)$ is odd, then none of the $r_i$ is fixed by the $I_{\mathbb{Q}_p}$-action.
\end{lemma}
\begin{proof}
By Lemma \ref{lem: fd splitting field}, the splitting field of $f_d$ over $\bQ$ is $\bQ(\alpha_d, \sqrt{d}, \sqrt{3})$, where the minimal polynomial of $\alpha_d$ is $x^3 - \frac{3}{4}x - \frac{d-3}{4(d+3)}$.
Since $d \equiv 0 \pmod{p}$, we have
\begin{align*}
    x^3 - \frac{3}{4}x - \frac{d-3}{4(d+3)}\equiv x^3-\frac{3}{4}x+\frac{1}{4}=(x+1)\left(x-\frac{1}{2}\right)^2 \pmod{p}.
\end{align*}
Replacing $\alpha_d$ by a Galois conjugate if necessary, Hensel's lemma yields a factorization over $\mathbb{Q}_p$:
\begin{align*}
    x^3-\frac{3}{4}x-\frac{d-3}{4(d+3)}=(x-\alpha_d)\left(x^2+\alpha_d x+\alpha_d^2-\frac{3}{4}\right).
\end{align*}
It follows that $\alpha_d \in \bQ_p$, and the splitting field of $f_d$ over $\bQ_p$ is $\bQ_p(\sqrt{d}, \sqrt{3})$.
Since $p \nmid 3$, the inertia group $I_{\bQ_p}$ action factors through the quotient $I_{\mathbb{Q}_p}\rightarrow\Gal(\bQ_p(\sqrt{d})/\bQ_p)$.

The proof of Lemma \ref{lem: quad split of fd} also works for $\bQ_p$, so it gives a factorization over $\bQ_p(\alpha_d, \sqrt{d}) = \bQ_p(\sqrt{d})$ and $\bQ_{p}(\alpha_d) = \bQ_p$.
Without loss of generality, let $r_1 + r_2 = -a$, $r_3 + r_4 = -b$, and $r_5 + r_6 = -c$, and let $\sigma_d$ be a lifting of non-trivial element in $\Gal(\bQ_p(\sqrt{d})/\bQ_p)$.
Then the factorization of Lemma \ref{lem: quad split of fd} shows that 
$\sigma_d$ defines an automorphism on $\{r_1,r_2\}$ and a bijection between $\{r_3,r_4\}$ and $\{r_5,r_6\}$.
We may assume that $\sigma_d(r_3)=r_6$ and $\sigma_d(r_4)=r_5$.

Moreover, the factorization (\ref{eqn: fd two cubic factor}) which works over $\bQ_p(\sqrt{3d})$ shows that $r_3$ and $r_5$ (resp. $r_4$ and $r_6$) are in the same cubic factor. This shows that $r_1$ and $r_2$ are not in the same factor so that $\sigma_d(r_1) = r_2$. Therefore, $\sigma_d$ gives $3$ orbits of $\cR$ of size $2$.
\end{proof}

\begin{proposition} \label{prop: cond exp p|d}
Let $d$ be a square-free integer and let $p \geq 5$ be a prime such that $p \mid d$.
Then, the conductor exponent of $C_d$ at $p$ is $2$.    
\end{proposition}
\begin{proof}
We first recall that it has tame reduction by Lemma \ref{lem: tame p>3}.
From the definition and Lemma \ref{lem: cluster p|d},
\begin{align*}
    d_{\cR} = \frac{\ord_p(d)}{2}, \qquad 
    \widetilde{\lambda}_{\cR} = \frac{1}{2} \lbrb{\ord_p\lbrb{\frac{27(d-3)}{d+3}} + 6d_{\cR} + 0 } = \frac{3}{2} \ord_p(d).
\end{align*}
Since $I_{\cR} = I_{\bQ_p}$ and $\xi_{\cR}(a) = \max \lcrc{-\ord_2([I_{\bQ_p} : I_{\cR}]a), 0}$,
\begin{align*}
    U &= \lcrc{\fs \neq \cR \text{ odd} : \xi_{\cR} \lbrb{ \frac{3\ord_p(d) }{2}  }  \leq \xi_{\cR} \lbrb{\frac{\ord_p(d)}{2}}  } = \lcrc{ \lcrc{r} : r \in \cR }.
\end{align*}
By Lemma \ref{lem: inertia action p|d}, we have $|U/I_{\bQ_p}| = 3$.

For $V$, the only proper cluster is $\cR$ itself, which is not \"ubereven.
Since $\ord_p(d) = 1$, we have $\xi_{\cR}(\widetilde{\lambda}_{\cR}) = 1$ and hence $V$ is the empty set.
From
\begin{align*}
    |\mathcal{R}|=6,\quad\ord_p\left(\frac{27(d-3)}{d+3}\right)=0
\end{align*}
and Theorem \ref{thm: BBB12.3}, we get $ n_\mathrm{tame}=4-3+0+1=2.$
\end{proof}

\begin{lemma} \label{lem: cluster picture p|d+3}
Let $p \geq 5$ be a prime such that $\ord_p(d+3) > 0$.
Then the cluster picture of $f_d$ over $\bQ_p$ is $\lcrc{\bullet, \cdots, \bullet}_{\frac{\ord_p(d+3)}{3}}$.   
\end{lemma}
\begin{proof}
Since $d\equiv-3\bmod p$, we have
\begin{align*}
\frac{d+3}{27(d-3)}f_d(x)\equiv(x-1)^6\pmod{p}.
\end{align*}
The Newton polygon of
\begin{align*}
\frac{d+3}{27(d-3)}f_d(x+1)&=x^6-\frac{6(d+3)}{d-3}x^5-\frac{5(d+3)^2}{d-3}x^4-\frac{20(d+3)^2}{3(d-3)}x^3+\frac{5}{3}(d+3)^2x^2\\
&\quad\quad+\frac{2(d-1)(d+3)^2}{d-3}x-\frac{(d+3)^2(d^2-18d+9)}{27(d-3)}
\end{align*}
over $\mathbb{Q}_p$ is the line segment connecting $(0,2\ord_p(d+3))$ and $(6,0)$. Hence
\begin{align*}
\ord_p(r_i-1)=\frac{\ord_p(d+3)}{3}.
\end{align*}
On the other hand,
\begin{align*}
\frac{1}{27(d-3)(d+3)}f_d\left(\sqrt[3]{d+3}x+1\right)\equiv x^6+\frac{4}{9}\pmod{\mathfrak{m}_{\mathbb{Q}_p(\sqrt[3]{d+3})}}
\end{align*}
which has discriminant $2^{16}3^{-4}$. Consequently, if $p>3$, then
\begin{align*}
\ord_p(r_i-r_j)&=\ord_p((r_i-1)-(r_j-1))\\
&=\ord_p\left(\frac{r_i-1}{\sqrt[3]{d+3}}-\frac{r_j-1}{\sqrt[3]{d+3}}\right)+\frac{\ord_p(d+3)}{3}=\frac{\ord_p(d+3)}{3}.
\end{align*}
So the assertion follows.
\end{proof}

\begin{proposition}  \label{prop: cond exp p|d+3}
    Let $d$ be a square-free integer and let $p\geq5$ be a prime such that $p\mid(d+3)$.
    \begin{enumerate}
        \item If $\ord_p(d+3)$ is odd, then the conductor exponent of $C_d$ is $4$.
        \item If $\ord_p(d+3)$ is even and $3\nmid\ord_p(d+3)$, then the conductor exponent of $C_d$ is $4$.
        \item If $6\mid\ord_p(d+3)$, then the conductor exponent of $C_d$ is $0$.
    \end{enumerate}
\end{proposition}
\begin{proof} We note that $C_d$ has tame reduction at $p \geq 5$.
By Lemma \ref{lem: cluster picture p|d+3} and definitions,
\begin{align*}
    d_\mathcal{R}&=\frac{\ord_p(d+3)}{3},\\
    \widetilde{\lambda}_\mathcal{R}&= \frac{1}{2}\left(-\ord_p(d+3) + 6d_{\cR} + 0\right)=\frac{\ord_p(d+3)}{2}.
\end{align*}

Since the Newton polygon of
\begin{align*}
    x^3-\frac{3}{4}x-\frac{d-3}{4(d+3)}
\end{align*}
over $\mathbb{Q}_p$ is the lower convex hull of $(0,-\ord_p(d+3))$, $(1,0)$, and $(3,0)$, we have
\begin{align*}
    \ord_p(\alpha_d)=-\frac{\ord_p(d+3)}{3}.
\end{align*}

Since $I_\mathcal{R}=I_{\mathbb{Q}_p}$, we have
\begin{align*}
    \xi_\mathcal{R}(d_\mathcal{R})=\max\left\{-\ord_2\left(\frac{\ord_p(d+3)}{3} \right), 0\right\}=0,
\end{align*}
\begin{align*}
    \xi_{\cR}(\widetilde{\lambda}_{\cR}) = \max \lcrc{-\ord_2\left(\frac{\ord_p(d+3)}{2} \right), 0} =  \left\{
    \begin{array}{ll}
    1     & \textrm{if $\ord_p(d+3)$ is odd,}  \\
    0     & \textrm{if $\ord_p(d+3) >0 $ is even.}  
    \end{array}
    \right.
\end{align*}

If $\ord_p(d+3)$ is odd, then $U$ and $V$ are empty. By Theorem \ref{thm: BBB12.3}, the conductor exponent is $4$.

Assume now that $\ord_p(d+3)$ is even. In this case, $U=\{\{r\}\ |\ r\in\mathcal{R}\}$ is the set of all singletons and $V=\{\mathcal{R}\}$. To apply Theorem \ref{thm: BBB12.3}, we have to compute the $I_{\mathbb{Q}_p}$-action on $U$. 

If $3\nmid\ord_p(d+3)$, then $\mathbb{Q}_p(\alpha_d)/\mathbb{Q}_p$ is totally ramified. It follows that $I_{\mathbb{Q}_p}$ acts on $U$ via
\begin{align*}
    \xymatrix{I_{\mathbb{Q}_p}=I_{\mathbb{Q}_p(\sqrt{d},\sqrt{3})} \ar[r] & \mathrm{Gal}(F_{d,p}/\mathbb{Q}_p(\sqrt{d},\sqrt{3})) \ar[r]^-\sim & C_3}
\end{align*}
where $F_{d,p}$ is the splitting field of $f_d$ over $\mathbb{Q}_p$. Now the factorization \eqref{eqn: fd two cubic factor} shows that $C_3$ acts transitively on each cubic factor, which implies that there are two $I_{\mathbb{Q}_p}$-orbits of $U$, each of size $3$. Consequently, $|U/I_{\mathbb{Q}_p}|=2$ and hence the conductor exponent becomes $4$ by Theorem \ref{thm: BBB12.3}.\

Finally, assume that $6\mid\ord_p(d+3)$. Denote $\mathrm{ord}_p(d+3)=3n$ for a positive integer $n$ so that $\mathrm{ord}_p(\alpha_d)=-n$. We have
\begin{align*}
    \left(\frac{x}{p^n}\right)^3-\frac{3}{4}\frac{x}{p^n}-\frac{d-3}{4(d+3)}=\frac{1}{p^{3n}}\left(x^3-\frac{3p^{2n}}{4}x-\frac{p^{3n}(d-3)}{4(d+3)}\right)
\end{align*}
so the monic cubic is a minimal polynomial of $p^n\alpha_d$. Taking modulo $p$, we get a polynomial of the form $x^3-c$ with $c\not\equiv0\bmod p$. Since $p>3$, this reduction cannot have multiple roots. Hence the discriminant of the monic cubic is a $p$-adic unit. This shows that $\mathbb{Q}_p(\alpha_d)/\mathbb{Q}_p$ is unramified and hence $I_{\mathbb{Q}_p}$ acts trivially on $U$. Therefore, the conductor exponent is $n_\mathrm{tame}=4-6+1+1=0$ by Theorem \ref{thm: BBB12.3}.
\end{proof}

\begin{proposition} \label{prop: bad ap}
Let $d$ be a square-free integer.
For a prime $p \geq 5$, we have 
$a_{p^i}(J_d) \ll \sqrt{p^i}$ if $p \| d$, and 
$a_{p^i}(J_d) = 0$ if $p \mid (d+3)$ but $p^6 \nmid (d+3)$.
\end{proposition}
\begin{proof}
We recall that the toric rank of an abelian variety is non-decreasing under finite base change (cf. \cite[Proposition 2.4]{CX08}).
Since $J_d$ is $L_d$-isogenous to $E_0^2$, whose toric rank is zero, the toric rank of $J_d$ over $\bQ$ must be zero.

Let $p \geq 5$ be a prime where $J_d$ has bad reduction. 
By Proposition \ref{prop: cond bound}, we only need to consider primes $p$ such that $p \mid d$ or $p \mid (d+3)$.
In any case, the conductor exponent at $p$ is $t + 2u + \delta$, where $t$ is the toric rank, $u$ is the unipotent rank, and $\delta$ is the Swan conductor, as we discussed in (\ref{eqn: cond exp sum}).
By Proposition \ref{prop: cond exp p|d}, $J_d$ has tame reduction at $p$, toric rank $0$, and the conductor exponent is $2$ when $p \mid d$.
So the abelian variety part of the Chevalley decomposition has dimension $1$, which gives the first claim.
When $p \mid (d+3)$ and $p^6 \nmid (d+3)$, we have $u = 2$ by Proposition \ref{prop: cond exp p|d+3}.
Since the unipotent group has no $\ell$-torsion, we have $\dim_{\bQ_{\ell}} (V_{\ell}J_d)^{I_{\bQ_p}} = 0$ in this case.
\end{proof}

\begin{remark}
Since $C_d$ has tame reduction at all primes $p > 3$, one can already obtain an explicit upper bound for the average analytic rank without determining the conductor precisely. 
Nevertheless, using the exact conductor exponents improves the bound (without it, we would obtain $24$ instead of $18$ in Theorem \ref{thm: AAR}).
Moreover, the cluster pictures in Lemma \ref{lem: cluster p|d} and Lemma \ref{lem: cluster picture p|d+3} explain why the twists $C_d$ exhibit a comparatively simple local structure, compared to that of a generic genus $2$ curve.
\end{remark}

\section{Main results} \label{sec: proof of main}

\subsection{Proof of the main theorem}
Throughout this section, we use the normalization of the $L$-functions introduced in section \ref{sec: pre}, replacing $s$ by $s+\frac{1}{2}$.
In particular, the central point is $s = \frac{1}{2}$.
We first recall the explicit formula following Iwaniec--Kowalski \cite{IK}.
Let $f$ be an object with root number $w(f)$, conductor $\ff(f)$, gamma factor $\gamma(f, s)$, and Dirichlet series
\begin{align*}
    L(f, s) = \sum_{n=1}^{\infty} \frac{\lambda_n(f)}{n^s}.
\end{align*}
Using the conductor and the gamma factor, the completed $L$-function is defined by
\begin{align*} 
    \Lambda(f, s) = \ff(f)^{\frac{s}{2}}\gamma(f, s) L(f, s),
\end{align*}
and it satisfies the functional equation 
\begin{align*}
    \Lambda(f, s) = w(f) \Lambda(\overline{f}, 1 - s).
\end{align*}
We define $\Lambda_n(f)$ as the constants satisfying
\begin{align*}
    -\frac{L'}{L}(f, s) = \sum_{n=1}^{\infty} \frac{\Lambda_n(f)}{n^s}.
\end{align*}
By section \ref{sec: pre}, for a genus $2$ curve $C$ that has good reduction at $p$, 
\begin{align*}
    \Lambda_{p^k}(C) = \lbrb{\sum_{i=1}^4 \lbrb{ \frac{\alpha_{p, i}}{\sqrt{p}} }^k}  \log p
    = b_{p^k}(C) \log p.
\end{align*}
Note that $b_p(C) = \lambda_p(C)$ and $b_{p^2}(C) \neq \lambda_{p^2}(C)$.

The following is an explicit formula from \cite[Theorem 5.12]{IK}.
\begin{proposition} \label{prop: explicit}
    Let $\phi$ be an even Schwartz function whose Fourier transform is compactly supported.
    Under the generalized Riemann hypothesis for $L(f, s)$, 
    \begin{align*}
        \sum_{\rho} \phi \lbrb{ \gamma \frac{\log X}{2 \pi}}
        &= \widehat{\phi}(0) \frac{\log \ff(f)}{\log X}
        - \frac{1}{\log X} \sum_n \lbrb{\frac{\Lambda_n(f) + \Lambda_n(\overline{f}) }{\sqrt{n}}} \widehat{\phi}\lbrb{\frac{\log n}{\log X}}\\
        &+ \frac{1}{2 \pi} \int^{\infty}_{-\infty} \lbrb{\frac{\gamma'}{\gamma}\left(f,  \frac{1}{2} + it\right) + \frac{\gamma'}{\gamma}\left(\overline{f},  \frac{1}{2} + it \right)} \phi\lbrb{\frac{t\log X}{2 \pi}} dt
    \end{align*}
    where $\rho = \frac{1}{2} + i\gamma$ is taken over non-trivial zeros of $L(f, s)$.
\end{proposition}

Recall the definition of $S(X)$ from the introduction.
Then
\begin{align*}
    |S(X)| = \frac{6}{\pi^2}X + O\lbrb{X^{\frac{1}{2}}}.
\end{align*}
Let
\begin{align*}
    S_1 :=\frac{2}{|S(X)|\log X}  \sum_{d \in S(X)}
    \sum_p \frac{b_{p}(C_d) \log p }{\sqrt{p}} \widehat{\phi}\lbrb{\frac{\log p}{\log X}}
\end{align*}
and
\begin{align*}
    S_2 := \frac{2}{|S(X)|\log X} \sum_{d \in S(X)} \sum_p \frac{b_{p^2}(C_d) \log p }{p} \widehat{\phi}\lbrb{\frac{2\log p}{\log X}}.
\end{align*}

We recall that $C_d$ over $\bQ$ has the analytic continuation to $\bC$ since it has complex multiplication.
We denote the analytic rank by $g_d$.

\begin{proposition}
Assume the generalized Riemann hypothesis for the family $\lcrc{C_d/\bQ : d \in S(X)}$. Let $\phi$ be a positive test function whose Fourier transform is compactly supported. Then,
\begin{align} \label{eqn: average rank bound}
    \frac{1}{|S(X)|} \sum_{d \in S(X)} g_d \leq (6+o(1))\frac{ \widehat{\phi}(0)}{\phi(0)} - \frac{1}{\phi(0)}S_1 - \frac{1}{\phi(0)}S_2 + O\lbrb{\frac{1}{\log X}}.
\end{align}
\end{proposition}
\begin{proof}
We note that $\gamma(f, s)$ is a product of an exponential function $\pi^{-\frac{ks}{2}}$ for some $k \in \bZ$ and the Gamma functions (cf. \cite[(5.3)]{IK}).
Therefore
\begin{align*}
    \frac{1}{2 \pi} \int^{\infty}_{-\infty} \lbrb{\frac{\gamma'}{\gamma}\left(f,  \frac{1}{2} + it\right) + \frac{\gamma'}{\gamma}\left(\overline{f},  \frac{1}{2} + it\right)} \phi\lbrb{\frac{t\log X}{2 \pi}} dt \ll \frac{\log \log X}{\log X}.
\end{align*}
On the other hand, the Weil bound for genus $2$ curves gives that $|\Lambda_{p^k}(C_d)|  \leq 4\log p$, which gives
\begin{align*}
    \frac{1}{\log X} \sum_{\substack{n = p^k \\ k \geq 3}} \lbrb{\frac{\Lambda_n(C_d) + \overline{\Lambda_n(C_d)} }{\sqrt{n}}} \widehat{\phi}\lbrb{\frac{\log n}{\log X}} \ll \frac{1}{\log X}.
\end{align*}

Since the test function $\phi$ is positive, 
\begin{align*}
\frac{1}{|S(X)|} \sum_{d \in S(X)} g_d \leq 
    \frac{1}{\phi(0)|S(X)|} \sum_{d \in S(X)} \sum_{\gamma_d} \phi \lbrb{\gamma_d \frac{\log X}{2 \pi}}
\end{align*}
where $\rho_d = \frac{1}{2} + i\gamma_{d}$ are non-trivial zeros of $L(C_d/\bQ, s)$.
By the explicit formula in Proposition \ref{prop: explicit}, 
\begin{align*}
    \sum_{\gamma_d} \phi \lbrb{ \gamma_d \frac{\log X}{2 \pi}}
    = \widehat{\phi}(0) \frac{\log \ff(C_d)}{\log X}
        &- \frac{2}{\log X} \sum_p \frac{\Lambda_p(C_d) }{\sqrt{p}} \widehat{\phi}\lbrb{\frac{\log p}{\log X}} \\
    &- \frac{2}{\log X} \sum_p \frac{\Lambda_{p^2}(C_d)}{p} \widehat{\phi}\lbrb{\frac{2\log p}{\log X}} + O \lbrb{\frac{1}{\log X}}.
\end{align*}
Hence, the average of analytic ranks is bounded by 
\begin{align*}
    &\frac{\widehat{\phi}(0)}{\phi(0)|S(X)|\log X} \sum_{d \in S(X)} \log \ff(C_d) \\
    &-\frac{2}{\phi(0)|S(X)|\log X} \sum_{d \in S(X)}
    \sum_p \frac{\Lambda_{p}(C_d)  }{\sqrt{p}} \widehat{\phi}\lbrb{\frac{\log p}{\log X}} \\
    &-\frac{2}{\phi(0)|S(X)|\log X} \sum_{d \in S(X)}
    \sum_p \frac{\Lambda_{p^2}(C_d) }{p} \widehat{\phi}\lbrb{\frac{2\log p}{\log X}} + O \lbrb{\frac{1}{\log X}}.
\end{align*}
By Proposition \ref{prop: cond exp p|d}, we have $\ff(C_d) \leq C d^2 (d+3)^4$, so 
\begin{align*}
    \sum_{d \in S(X)} \log \ff(C_d)
    &\leq   |S(X)| \log C + (6 + o(1))\sum_{d \in S(X)} \log d \\
    &\leq   |S(X)| \log C + (6 + o(1)) \frac{6}{\pi^2}X \log X + O(X)
\end{align*}
which gives
\begin{align*}
    \frac{\widehat{\phi}(0)}{\phi(0)|S(X)|\log X} \sum_{d \in S(X)} \log \ff(C_d)
    = (6+o(1))\frac{ \widehat{\phi}(0)}{\phi(0)} + O\lbrb{\frac{1}{\log X}}.
\end{align*}
Since $\Lambda_p(C_d) = \lambda_p(C_d) \log p$, we can recover $S_1$ and $S_2$.
\end{proof}

We recall that $\alpha_d$ is a zero of
\begin{align*}
    g_d(x) := x^3 - \frac{3}{4}x - \frac{d-3}{4(d+3)}.
\end{align*}

\begin{lemma} \label{lem: twisted d counting}
If $p\ge 5$ is a prime, then the number of $d\in \bF_p\setminus\lcrc{0,-3}$ with the prescribed data $(p\bmod 3,\ \quadsym{d}{p},\ \text{splitting type of }g_d)$
is summarized in the following Table \ref{tab:gd-splitting-mod3}.
\begin{center}
\begin{tabular}{c|c|c|c}
\hline
$p \bmod 3$ & $\quadsym{d}{p}$ & splitting type of $g_d$ in $\bF_p$ 
& $\#\{d\in \bF_p\setminus\lcrc{0,-3}\}$ \\ \hline
$1$ & $+1$ & split completely & $\frac{p-7}{6}$ \\
$1$ & $+1$ & irreducible      & $\frac{p-1}{3}$ \\
$1$ & $-1$ & linear $\times$ quadratic & $\frac{p-1}{2}$ \\ \hline
$2$ & $+1$ & split completely & $\frac{p-5}{6}$ \\
$2$ & $+1$ & irreducible      & $\frac{p+1}{3}$ \\
$2$ & $-1$ & linear $\times$ quadratic & $\frac{p-3}{2}$ \\ \hline 
\end{tabular}
\captionof{table}{}
\label{tab:gd-splitting-mod3}
\end{center}
\end{lemma}
\begin{proof}
We recall that the discriminant of $g_d$ is $\frac{81d}{4(d+3)^2}$.
Hence for $d \neq 0, -3 \pmod{p}$, 
\begin{align*}
    \quadsym{\mathrm{disc}(g_d)}{p} = \quadsym{d}{p}.
\end{align*}
Therefore, $d$ is non-square modulo $p$ if and only if $g_d$ is a product of linear and quadratic polynomials.
Since $-3$ is a square in $\mathbb{F}_p$ if and only if $p \equiv 1 \pmod{3}$, we have the third and sixth row.


Let $\kappa_d := \frac{d-3}{d+3}$ for $d \not\equiv -3 \pmod{p}$.
Then the map $d \mapsto \kappa_d$ gives a bijection between $\bF_p \setminus\lcrc{0, -3}$ and $\bF_p \setminus \lcrc{\pm 1}$.
Also, $g_d(t) = 0$ if and only if $4t^3 - 3t = \kappa_d$ unless $d \equiv -3 \pmod{p}$.
Since the $\bF_p$-points of $4x^3 - 3x = \pm 1$ is $\pm \frac{1}{2}$, and  $\pm 1$,
we have
\begin{align*}
    \# \lcrc{(t, d) \in \bF_p^2 : g_d(t) = 0, d \neq 0, -3} 
    &= \# \lcrc{(t, \kappa ) \in \bF_p^2 : 4t^3 - 3t = \kappa, \, \kappa \neq \pm 1 } \\
    &= p - \# \lcrc{t \in \bF_p : 4t^3 - 3t = \pm 1 } \\
    &= p - 4.
\end{align*}
Let $n_s, n_l$ be the number of $d \not\equiv 0, -3 \pmod{p}$ such that $g_d$ splits completely, is a product of linear and quadratic over $\bF_p$, respectively.
Then, the above computation shows that $3n_s +n_l = p-4. $
Hence, the results follow.
\end{proof}

We recall that $S(X)$ is the set of square-free integers $d$ satisfying $0 < d  < X$.
In this definition and the following lemma, we use $\overline{n}$ to denote the image of $n \in \bZ$ in $\bF_p$.
For a positive real number $X$ and a subset $A_p \subset \bF_p^\times$, we define
\begin{align*}
    c(A_p, X) := \left| \lcrc{n \in \lbrb{ p\lfloor \frac{X}{p} \rfloor, X  } : \overline{n} \in A_p } \right|.
\end{align*}
There is a naive bound $|c(A_p, X)| \leq |A_p|$. 

\begin{lemma} \label{lem:sqf sieve}
Let $p \geq 5$ be a prime and let $w: \bF_p \to \bC$ be a weight function on $\bF_p$ satisfying $w(0) = 0$.
Then,
\begin{align*}
    \sum_{\substack{ d \in S(X)}} w(\overline{d}) =
    \frac{1}{1 - p^{-2}}\frac{|S(X)|}{p} \sum_{a \in \bF_p^\times} w(a)  + O\lbrb{  \sqrt{X} \sum_{a \in \bF_p^\times}|w(a)| }.
\end{align*}
\end{lemma}
\begin{proof}
By the coordinate change $d = m^2k$, 
\begin{align*}
    \sum_{\substack{ d \in S(X) }} w(\overline{d})
    &= \sum_{\substack{d < X }} \mu^2(d)w(\overline{d})
    = \sum_{m < \sqrt{X}} \mu(m) \sum_{\substack{k < X/m^2}}w(\overline{m^2k}).
\end{align*}
Since $w(0) = 0$, the contribution of $p \mid m$ is zero.
On the other hand, when $p \nmid m$, multiplication by $m^2$ induces a permutation on $\bF_p^\times$.
Therefore, 
\begin{align*}
    \sum_{k \leq Y } w(\overline{m^2k}) = \frac{Y}{p} \sum_{a \in \bF_p^\times}w(a) + O\lbrb{ \sum_{a \in \bF_p^\times}|w(a)| }.
\end{align*}
Together with
\begin{align*}
    \sum_{p \nmid m} \frac{\mu(m)}{m^2} = \frac{1}{\zeta(2)} \frac{1}{1 - p^{-2}},
\end{align*}
we obtain the conclusion.
\end{proof}

When we pick $w$ as a characteristic function on $A_p \subset \bF_p^\times$, we have an error term $|A_p|\sqrt{X}$.
We note that the contribution at $0 \in \bF_p$ differs from that of elements in $\bF_p^\times$.
Similar phenomena have been observed in \cite[Proposition 2.3]{CJ1}.

Using the various prime number theorems, we can bound the weighted sum of traces of Frobenius.

\begin{lemma} 
Let $E_0 : y^2 = x^3+1$ be the elliptic curve.
Let $\widehat{\phi}$ be an even absolutely continuous function supported in $[-\sigma, \sigma]$ which is the Fourier transform of $\phi$.
Then we have
\begin{align}
    \frac{1}{\log X}\sum_{p \equiv 1 \, (3)} \frac{\log p}{p^2} \widehat{\phi} \lbrb{\frac{2 \log p }{\log X}} (a_p(E_0)^2 - 2 p) &= O\lbrb{\frac{1}{\log X}}. \label{eqn: ap^2 weight E0-2p} 
\end{align}
\end{lemma}
\begin{proof}


Since $E_0$ has complex multiplication by $\bQ(\sqrt{-3})$, there is a Hecke character $\psi_0$ such that
\begin{align*}
a_p(E_0) = \psi_0(\fp) + \psi_0(\overline{\fp}), \qquad \sqrt{p} = |\psi_0(\fp)|  
\end{align*}
where $\fp$ is a prime ideal of the ring of integers of $\bQ(\sqrt{-3})$ when $p \equiv 1 \pmod{3}$ (cf. \cite[Corollary 10.4.1]{Sil2}).
We also note that $a_p(E_0) = 0$ if $p \equiv 2 \pmod{3}$.
Since
\begin{align*}
    a_p(E_0)^2 - 2p = (\psi_0(\fp) + \psi_0(\overline{\fp}))^2 - 2 |\psi_0(\fp)|^2
    = \psi_0(\fp)^2 + \overline{\psi_0}(\fp)^2,
\end{align*}
we have
\begin{align*}
\frac{1}{\log X}\sum_{p \equiv 1 \, (3)} \frac{\log p}{p^2} \widehat{\phi} \lbrb{\frac{2 \log p }{\log X}} (a_p(E_0)^2 - 2 p)= 
    \frac{1}{\log X} \sum_{\substack{f(\fp|p) = 1 }} \frac{\log N(\fp)}{N(\fp)^2} \widehat{\phi} \lbrb{\frac{2 \log N(\fp) }{\log X}} \psi_0(\fp)^2 .
\end{align*}
Here $N$ is the norm map from $\bQ(\sqrt{-3})$ to $\bQ$.
We consider
\begin{align*}
    \vartheta_{\psi_0^2}(t) := \sum_{\substack{N(\fp) \leq t \\ f(\fp \mid p) = 1}} \frac{\log N(\fp)}{N(\fp)} \psi_0(\fp)^2
\end{align*}
which satisfies the prime number theorem
\begin{align*}
    \vartheta_{\psi_0^2}(t) = O\lbrb{t \exp (-c \sqrt{\log t})}
\end{align*}
for some $c>0$.
By the usual partial summation argument (cf. \cite[Lemma 5.5]{CJP}), we have
\begin{align*}
    &\frac{1}{\log X}\sum_{\substack{ f(\fp| p) = 1}} \frac{\log N(\fp) }{ N(\fp)^2 } \widehat{\phi}\lbrb{\frac{ 2 \log N(\fp)}{ \log X}} \psi_0(\fp)^2 \\
    &= \int^{X^{\frac{\sigma}{2}}}_1 \frac{1}{ t \log X} \widehat{\phi}\lbrb{\frac{2\log t}{\log X}} d\vartheta_{\psi_0^2}(t)  \\
    &= -\int^{X^{\frac{\sigma}{2}}}_1 \vartheta_{\psi_0^2}(t) \frac{d}{dt} \lbrb{\frac{1}{ t \log X} \widehat{\phi} \lbrb{\frac{2 \log t}{\log X}}  } dt + O\lbrb{\frac{1}{\log X}} \\
    &= -\int^{X^{\frac{\sigma}{2}}}_1 \vartheta_{\psi_0^2}(t) 
    \lbrb{  -\frac{1}{t^{2}\log X} \widehat{\phi}\lbrb{\frac{2\log t}{\log X}} + \frac{2}{t^{2} \log^2 X} \widehat{\phi}' \lbrb{\frac{2\log t}{\log X}}  }
    dt + O\lbrb{\frac{1}{\log X}} \\
    & \ll \frac{1}{\log X}.
\end{align*}
Hence, we have the conclusions.
\end{proof}

\begin{proposition} \label{prop: S1}
Suppose that $\mathrm{supp}(\widehat{\phi}) \subset [-\sigma, \sigma]$ for $\sigma > 0$.
Then,
\begin{align*}
    S_1  \ll \frac{1 + X^{\frac{3\sigma}{2} - \frac{1}{2}}}{\log X}.
\end{align*}
\end{proposition}
\begin{proof} 
Since $b_p(C) = \lambda_p(C)$, we have
\begin{align*}
    S_1= \frac{2}{|S(X)|\log X}  
    \sum_p  \widehat{\phi}\lbrb{\frac{\log p}{\log X}} \frac{\log p}{\sqrt{p}}
    \sum_{d \in S(X)} \lambda_p(C_d).
\end{align*}
By Proposition \ref{prop: list Lp} we have $\lambda_p(C_d) = 0$ if $p \equiv 2 \pmod{3}$ for a good prime $p$.
When $p \mid d$, we have $\lambda_p(C_d) \ll 1$ by Proposition \ref{prop: bad ap}.
Therefore, the contribution of prime $p \mid d$ is
\begin{align*}
    \ll \frac{2}{|S(X)|\log X}  
    \sum_p  \widehat{\phi}\lbrb{\frac{\log p}{\log X}} \frac{\log p}{\sqrt{p}}
    \sum_{\substack{d \in S(X) \\ p \mid d}} 1
    \ll \frac{2}{\log X} \sum_{p < X^{\sigma}} \frac{\log p}{p^{\frac{3}{2}}} \ll \frac{1}{\log X}.
\end{align*}
For the prime $p \mid (d+3)$, we have $\lambda_{p}(C_d) = 0$ if $p^6 \nmid (d+3)$ and $\lambda_p(C_d) \ll 1$ if $p^6 \mid (d+3)$, by Proposition \ref{prop: bad ap}. 
Therefore the contribution of $p \mid (d+3)$ is
\begin{align*}
    \ll \frac{2}{|S(X)|\log X}  
    \sum_p  \widehat{\phi}\lbrb{\frac{\log p}{\log X}} \frac{\log p}{\sqrt{p}}
    \sum_{\substack{d \in S(X) \\ p^6 \mid d + 3}} 1 \ll \frac{1}{\log X}.
\end{align*}

By Proposition \ref{prop: Kd and Cd notwist}, $f_K(p)$ is one of $1, 2, 3$, and $f_M(p) = 1$ if $p \equiv 1 \pmod{3}$.
Hence,
\begin{align*}
    S_1 &= \frac{2}{|S(X)|\log X}  \sum_{p\equiv 1 \,(3)}  \widehat{\phi}\lbrb{\frac{\log p}{\log X}} \frac{\log p}{\sqrt{p}}
    \sum_{d \in S(X)} \lambda_p(C_d) \\
    &= \frac{2}{|S(X)|\log X}  \sum_{p\equiv 1 \,(3)}  \widehat{\phi}\lbrb{\frac{\log p}{\log X}} \frac{\log p}{p} \\
    & \qquad \times
    \lbrb{ \sum_{\substack{d \in S(X) \\ I(p) = (f_L(p),1,1) \\ p \nmid d, d+3}} a_p(C_d)
    + \sum_{\substack{d \in S(X) \\ I(p) = (f_L(p),2,1)\\ p \nmid d, d+3}} a_p(C_d)
    + \sum_{\substack{d \in S(X) \\ I(p) = (f_L(p),3,1)\\ p \nmid d, d+3}} a_p(C_d) } + O\lbrb{\frac{1}{\log X} }.
\end{align*}
Here $f_L(p)$ is $f_K(p)$ or $2f_K(p)$.
Recall that $a_{p}(C_d) = -a_{p, 1}(C_d)$.
By Proposition \ref{prop: list Lp}, the summation in the bracket is 
\begin{align} \label{eqn: big bracket}
a_p(E_0) \lbrb{
    2 \sum_{\substack{d \in S(X) \\ f_L(p) = 1 \\ f_K(p) = 1\\ p \nmid d, d+3}}1
    - 2 \sum_{\substack{d \in S(X) \\ f_L(p) = 2 \\ f_K(p) = 1 \\ p \nmid d, d+3}}1
    - \sum_{\substack{d \in S(X) \\ f_L(p) = 3 \\ f_K(p) = 3 \\ p \nmid d, d+3}}1
    + \sum_{\substack{d \in S(X) \\ f_L(p) = 6 \\ f_K(p) = 3 \\ p \nmid d, d+3}}1
    }.
\end{align}
By Proposition \ref{prop: L bound}, we have $L = K(\sqrt{6(d+3)})$.
Therefore, 
\begin{align} \label{eqn: fL fK}
    f_L(p) = \begin{cases}
        f_K(p) & \textrm{if } \quadsym{6(d+3)}{p} = 1  \textrm{ or } 2 \mid f_K(p), \\
        2f_K(p) & \textrm{if } \quadsym{6(d+3)}{p} \neq 1 \textrm{ and } 2\nmid f_K(p).
    \end{cases} 
\end{align}
Hence, 
\begin{align*}
    \eqref{eqn: big bracket}
    = a_p(E_0) \lbrb{ 2 \sum_{\substack{d \in S(X) \\ f_K(p) = 1 \\ p \nmid d, d+3}} \quadsym{6(d+3)}{p}
    - \sum_{\substack{d \in S(X) \\ f_K(p) = 3 \\ p \nmid d, d+3}} \quadsym{6(d+3)}{p}}.
\end{align*}
To use Lemma \ref{lem:sqf sieve}, we choose the weight function
\begin{align*}
    w(\overline{d}) := 
    \begin{cases}
        2 \quadsym{6(d+3)}{p} & \textrm{if } f_K(p) = 1, d \not\equiv 0, -3 \pmod{p}, \\
        -\quadsym{6(d+3)}{p} & \textrm{if } f_K(p) = 3, d \not\equiv 0, -3 \pmod{p}, \\
        0 & \textrm{otherwise}.
    \end{cases}
\end{align*}
Let $N_d$ be the number of $\bF_p$-roots of $g_d$ so that
\begin{align*}
    N_d = \# \lcrc{t \in \bF_p : 4t^3 - 3t = \frac{d-3}{d+3} }.
\end{align*}
Since $K = \bQ(\alpha_d, \sqrt{d}, \sqrt{-3})$, $g_d$ is the irreducible polynomial of $\alpha_d$, and $p \equiv 1 \pmod{3}$, we have
\begin{align*}
    N_d = \begin{cases}
        3 & \textrm{if } f_K = 1, \\
        1 & \textrm{if } f_K = 2, \\
        0 & \textrm{if } f_K = 3. 
    \end{cases}
\end{align*}
Hence, 
\begin{align*}
    \sum_{\overline{d} \in \bF_p^\times } w(\overline{d}) = \sum_{d \in \bF_p \setminus \lcrc{0, -3}} \quadsym{6(d+3)}{p}(N_d - 1).
\end{align*}
Using the bijection $d \mapsto \kappa_d := \frac{d-3}{d+3}$ in Lemma \ref{lem: twisted d counting}, 
\begin{align*}
    d+3 = 3 \lbrb{ \frac{1+\kappa}{1-\kappa} + 1} = \frac{6}{1 - \kappa}.
\end{align*}
Then, we have
\begin{align*}
    \sum_{\overline{d} \in \bF_p^\times } w(\overline{d}) 
    &= \sum_{ \kappa \in \bF_p \setminus \lcrc{\pm 1}} \quadsym{1-\kappa}{p}(N_d - 1)
    = \sum_{\kappa \neq \pm 1 } \sum_{t \in \bF_p} \quadsym{1-\kappa}{p } 1_{4t^3 - 3t = \kappa} - \sum_{\kappa \neq \pm 1} \quadsym{1-\kappa}{p} \\
    & = \sum_{\substack{t \in \bF_p \\ 4t^3 - 3t \neq \pm 1 }} \quadsym{1-(4t^3 - 3t) }{p } - \sum_{\kappa \neq \pm 1} \quadsym{1-\kappa}{p} \\
    & = \sum_{\substack{t \in \bF_p \\ 4t^3 - 3t \neq \pm 1 }} \quadsym{1-t }{p } - \sum_{\kappa \neq \pm 1} \quadsym{1-\kappa}{p} + O(1)
\end{align*}
which is $O(1)$.
Hence by Lemma \ref{lem:sqf sieve}, we have
\begin{align*}
    \eqref{eqn: big bracket} \ll \frac{|S(X)|}{\sqrt{p}} + p^{\frac{3}{2}}\sqrt{X}.
\end{align*}
Therefore, 
\begin{align*}
    S_1 &\ll \frac{2}{|S(X)|\log X}  \sum_{\substack{p \leq X^{\sigma} \\ p\equiv 1 \,(3)}} \frac{\log p}{p} \lbrb{\frac{|S(X)|}{\sqrt{p}} + p^{\frac{3}{2}}\sqrt{X}} + O\lbrb{\frac{1}{\log X}} \\
    &\ll \frac{1}{X \log X} \sum_{\substack{p \leq X^{\sigma} \\ p\equiv 1 \,(3)}} \sqrt{p} \log p + O \lbrb{\frac{1}{\log X}} \\
    & \ll \frac{1 + X^{\frac{3\sigma}{2} - \frac{1}{2}} }{\log X}.
\end{align*}
Hence, we obtain the conclusion.
\end{proof}

\begin{proposition} \label{prop: S2}
Suppose that $\mathrm{supp}(\widehat{\phi}) \subset [-\sigma, \sigma]$ for $\sigma > 0$.
Then,
\begin{align*}
    S_2 = - \frac{1}{2} \phi(0) + O\lbrb{\frac{1}{\log X} \lbrb{1 + X^{\frac{\sigma}{2} - \frac{1}{2}}}}
\end{align*}
\end{proposition}
\begin{proof}
By Weil bound, we have $|b_{p^2}(J_d)| \leq 4$.
Hence, by a similar argument to the proof of Proposition \ref{prop: S1}, one can prove that the contribution of $p \mid d$ and $p \mid (d+3)$ is negligible.
By Lemma \ref{lem: fd splitting field} and Proposition \ref{prop: L bound}, we have $\Gal(L/\bQ) \cong S_3 \times C_2$ or $S_3 \times C_2^2$.
Therefore, $I(p)$ cannot be $(4, 2, 1)$, $(4, 2, 2)$, or $(12, 6, 2)$ since there is no element of order $4$ or $12$ in $\Gal(L/\bQ)$.

We suppose first that $p \equiv 1 \pmod{3}$.
By \eqref{eqn: fL fK} and Proposition \ref{prop: list Lp}, we have
\begin{align*}
    b_{p^2}(C_d) = 
    \begin{cases}
        2 \lbrb{\frac{a_p(E_0)^2}{p} - 2} & \textrm{if } f_K(p) = 1, 2 \\
        - \lbrb{\frac{a_p(E_0)^2}{p} - 2} & \textrm{if } f_K(p) = 3.
    \end{cases}
\end{align*}
We recall that $K$ depends on $d$. When $p \equiv 1 \pmod{3}$, let
\begin{align*}
    N_{j, 1} := | \lcrc{d \in S(X) : p \nmid d(d+3), f_K(p) = j}|.
\end{align*}
Similarly we define $N_{j, 2}$ for the prime $p \equiv 2 \pmod{3}$.
Then the summand over $p \equiv 1 \pmod{3}$ in $S_2$ with $p \nmid d(d+3)$ is
\begin{align*}
    &\frac{2}{|S(X)|\log X} \sum_{d \in S(X)} \sum_{\substack{p \equiv 1 \, (3) \\ p \nmid d(d+3)}} \frac{ \log p }{p} b_{p^2}(C_d)\widehat{\phi}\lbrb{\frac{2\log p}{\log X}} \\
    &=\frac{2}{|S(X)|\log X} \sum_{p \equiv 1 \, (3)}  \frac{ \log p }{p}   \widehat{\phi}\lbrb{\frac{2\log p}{\log X}}  \lbrb{\frac{a_p(E_0)^2}{p} - 2} \lbrb{2N_{1,1} + 2N_{2,1} - N_{3,1} }.
\end{align*}
By Lemma \ref{lem: twisted d counting} and Lemma \ref{lem:sqf sieve}, we have
\begin{align*}
    N_{j, 1} = \frac{1}{1-p^{-2}} \frac{c_{j, 1}}{p} |S(X)| + O\lbrb{p\sqrt{X}}
\end{align*}
for $c_{1,1}  = \frac{p-7}{6}, c_{2, 1} = \frac{p-1}{2}, c_{3, 1} = \frac{p-1}{3}$.
Therefore, the sum for $p \equiv 1 \pmod{3}$ is
\begin{align*}
    \frac{2}{\log X} \sum_{p \equiv 1 \, (3)}  \frac{ \log p }{p}   \widehat{\phi}\lbrb{\frac{2\log p}{\log X}}  \lbrb{\frac{a_p(E_0)^2}{p} - 2}\lbrb{1 + O\lbrb{\frac{1}{p} + \frac{p}{\sqrt{X}} } }
\end{align*}
The main term part is $O\lbrb{\frac{1}{\log X}}$ by \eqref{eqn: ap^2 weight E0-2p}.
The error term gives
\begin{align*}
    \ll \frac{1}{\sqrt{X}\log X} \sum_{p \leq X^{\frac{\sigma}{2}}} \log p \ll \frac{X^{\frac{\sigma}{2} - \frac{1}{2}}}{\log X}.
\end{align*}

Let $p \equiv 2 \pmod{3}$.
Then, $f_M(p) = 2$, $f_K(p)$ is $2$ or $6$, and $f_L(p) = f_K(p)$ by \eqref{eqn: fL fK}.
Hence, by Proposition \ref{prop: list Lp}, 
\begin{align*}
    b_{p^2}(C_d) = \begin{cases}
        -4, & \textrm{if } f_K(p) =2, \\
        2, & \textrm{if } f_K(p) =6.
    \end{cases}
\end{align*}
Then the summand over $p \equiv 2 \pmod{3}$ in $S_2$ with $p \nmid d(d+3)$ is
\begin{align*}
    \frac{2}{|S(X)|\log X} \sum_{p \equiv 2 \, (3)}  \frac{ \log p }{p}   \widehat{\phi}\lbrb{\frac{2\log p}{\log X}} \lbrb{-4N_{2,2} + 2N_{6, 2} },
\end{align*}
where
\begin{align*}
    N_{j, 2} = \frac{1}{1 - p^{-2}} \frac{c_{j, 2}}{p} |S(X)| + O(p\sqrt{X})
\end{align*}
for $c_{2, 2} = \frac{2p-7}{3}$ and $c_{6, 2} = \frac{p+1}{3}$ by Lemma \ref{lem: twisted d counting}.
Again by Lemma \ref{lem:sqf sieve}, the sum over $p \equiv 2 \pmod{3}$ is
\begin{align*}
    \frac{2}{|S(X)|\log X} \sum_{p \equiv 2 \, (3)}  \frac{ \log p }{p}   \widehat{\phi}\lbrb{\frac{2\log p}{\log X}} \lbrb{ -2 |S(X)| + O \lbrb{\frac{X}{p} + p\sqrt{X} }  }.
\end{align*}
By the prime number theorem in arithmetic progression, we have
\begin{align*}
    -\frac{4}{\log X} \sum_{p \equiv 2 \, (3)} \frac{\log p}{p} \widehat{\phi} \lbrb{\frac{2 \log p}{ \log X}} = -\frac{\phi(0)}{2} + O\lbrb{\frac{1}{\log X}}.
\end{align*}
Hence, we have the desired consequence.
\end{proof}

\begin{proof}[Proof of Theorem \ref{thm: AAR}]
We choose a test function
\begin{align*}
    \phi(x) = \frac{\sin^2(2 \pi x \frac{1}{2}\sigma )}{(2 \pi x)^2}
\end{align*}
whose Fourier inversion is
\begin{align*}
    \widehat{\phi}(u) = \frac{1}{2}\lbrb{\frac{1}{2}\sigma - \frac{1}{2}|u|}.
\end{align*}
Hence the support of $\widehat{\phi}$ is $\left[ - \sigma , \sigma \right]$, and
we note that $\frac{\widehat{\phi}(0)}{\phi(0)} = \frac{1}{\sigma}$.
By Proposition \ref{prop: S1} and Proposition \ref{prop: S2},  (\ref{eqn: average rank bound}) is
\begin{align*}
\frac{1}{|S(X)|} \sum_{d \in S(X)} g_d 
    &\leq (6+o(1))\frac{ \widehat{\phi}(0)}{\phi(0)}
    - \frac{1}{\phi(0)}S_1 - \frac{1}{\phi(0)}S_2
    + O\lbrb{\frac{1}{\log X}} \\
    &= \frac{1}{2} +   (6+o(1))\frac{1}{\sigma} + O\lbrb{ \frac{1 + X^{\frac{3\sigma}{2}- \frac{1}{2}} + X^{\frac{\sigma}{2} - \frac{1}{2} }   }{\log X}    }.
\end{align*}
By taking $\sigma = \frac{1}{3} - \epsilon$ for arbitrary $\epsilon > 0$, we have an upper bound $\frac{1}{2} + 18$.
\end{proof}

\subsection{Type B} \label{sec: other types}

Recall that there are $6$ conjugacy classes $D_{12}^A, \cdots, D_{12}^F$ in $C_2 \times D_{12}$, which are isomorphic to $D_{12}$ as an abstract group.
When the base field is $\bQ$, the type $D_{12}^E$ does not appear, and the other types have a non-trivial Brauer group condition except $D_{12}^B$ by \cite[\S 7]{CL}.
For the type $D_{12}^B$, the parameter $(u, v)$ is $(d, -3)$ by \cite[\S 7]{CL}.
In this case, the equation 
\begin{align*}
    u^3 - z^2 = 3s^2v
\end{align*}
in Definition \ref{def:uvsz} has a solution
\begin{align*}
    z = \frac{d(d+1)}{2}, \qquad s = \frac{d(d-1)}{6}.
\end{align*}
By Theorem \ref{thm:CLProp34}, we have a parametrization of twists of $D_{12}^B$-type of $C_0$, which is $y^2 = f_d(x)$ where
\begin{align*}
    f_d(x) &= \frac{27(d^2+d)}{2}x^6 + 81(d^2-d)x^5 + \frac{405(d^2+d)}{2}x^4 +270(d^2-d)x^3\\
    & + \frac{405(d^2+d)}{2}x^2 + 81(d^2-d)x + \frac{27(d^2 + d)}{2}.
\end{align*}

However, the Jacobian of the curve is not simple.

\begin{proposition} \label{lem: not simple}
Let $d \neq 0, \pm 1$ be an integer and let $C_d/\bQ$ be the curve defined by the equation $y^2 = f_d(x)$.
Then, the Jacobian of the curve $C_d$ is $\bQ$-isogenous to $E_{d^4} \times E_{d^5}$ where
\begin{align*}
    E_{d^4} : y^2 = x^3 + 2^33^3d^4, \qquad
    E_{d^5} : y^2 = x^3 + 2^33^3d^5.
\end{align*}
\end{proposition}
\begin{proof}
We can rewrite $f_d$ as follows:
\begin{align*}
    f_d=\frac{27d}{2}\left((x-1)^6+d(x+1)^6\right).
\end{align*}
The coordinate change
\begin{align*}
    t=\frac{x-1}{x+1},\quad u=\frac{y}{(x+1)^3}
\end{align*}
yields another defining equation for $C_d$:
\begin{align*}
    u^2=\frac{27d}{2}(t^6+d).
\end{align*}
Then we have surjections onto elliptic curves over $\mathbb{Q}$ as follows:
\begin{align*}
    \xymatrix{\pi_1:C_d \ar[r] & \displaystyle\left\{u^2 = \frac{27d}{2}(t^3 + d)\right\} & (t,u) \ar@{|->}[r] & (t^2,u) }
\end{align*}
\begin{align*}
    \xymatrix{\pi_2:C_d \ar[r] & \displaystyle\left\{u^2 = \frac{27d}{2}(t^3 + d^2)\right\} & (t,u) \ar@{|->}[r] & \displaystyle\left(\frac{d}{t^2},\frac{du}{t^3}\right)\rlap{\ .}}
\end{align*}
The former is isomorphic over $\mathbb{Q}$ to $E_{d^4}:y^2=x^3+2^{3}3^{3}d^{4}$ and the latter to $E_{d^5}:y^2=x^3+2^{3}3^{3}d^{5}$. Hence, $\Jac(C_d)$ is isogenous to $E_{d^4}\times E_{d^5}$.
\end{proof}

\end{document}